\documentclass[a4paper]{article}

\usepackage{ytableau}

\usepackage[margin=1in]{geometry} 
\usepackage{subcaption}
\usepackage{amsmath, amsthm, amssymb, xcolor, booktabs}

\usepackage[utf8]{inputenc}
\usepackage{hyperref}
\hypersetup{
	unicode,
   linkbordercolor=white,
	pdfauthor={Maria Teresa Arias},
	pdftitle={Scattering Networks on Finite Groups},
	pdfsubject={Harmonic Analysis and Machine Learning Article},
	pdfkeywords={scattering, machine learning, equivariance, groups},
}

\usepackage{amsmath}
\usepackage{amssymb}
\newcommand{\RR}{\mathbb{R}}

\newcommand{\ZZ}{\mathbb{Z}}
\newcommand{\CC}{\mathbb{C}}

\newcommand{\Cl}{\mathscr{C}}

\usepackage{amsthm}
\theoremstyle{plain}
\newtheorem{definition}{Definition} [section]
\newtheorem{corollary}[definition]{Corollary}
\newtheorem{proposition}[definition]{Proposition}
\newtheorem{theorem}[definition]{Theorem}
\newtheorem{lemma}[definition]{Lemma}
\newtheorem{example}[definition]{Example}
\newtheorem{remark}[definition]{Remark}

\usepackage{graphicx, color}
\graphicspath{{fig/}}

\usepackage{algorithm, algpseudocode} 
\usepackage{mathrsfs} 

\title{Scattering Networks on Noncommutative Finite Groups}
\author{Maria Teresa Arias, Davide Barbieri, Eugenio Hernández}

\date{\today
}

\begin{document}
	\maketitle
	
	\begin{abstract}
 Scattering Networks were initially designed to elucidate the behavior of early layers in Convolutional Neural Networks (CNNs) over Euclidean spaces and are grounded in wavelets. In this work, we introduce a scattering transform on an arbitrary finite group (not necessarily abelian) within the context of group-equivariant convolutional neural networks (G-CNNs). We present wavelets on finite groups and analyze their similarity to classical wavelets. We demonstrate that, under certain conditions in the wavelet coefficients, the scattering transform is non-expansive, stable under deformations, preserves energy, equivariant with respect to left and right group translations, and, as depth increases, the scattering coefficients are less sensitive to group translations of the signal, all desirable properties of convolutional neural networks. Furthermore, we provide examples illustrating the application of the scattering transform to classify data with domains involving abelian and nonabelian groups.
\\
		\noindent\textbf{Keywords:} scattering, machine learning, CNNs, equivariance, groups.
	\end{abstract}

	\tableofcontents

\section{Introduction}
\label{sec:intro}
Many interesting problems involve analyzing and manipulating structured data, often real-valued functions defined on domain sets with some structure. This is the case of image recognition, where convolutional neural networks (CNNs) have been extremely successful, giving rise to many generalizations and to a great interest in understanding how they work.

The concept of \textit{scattering transform} was originally introduced in \cite{Mallat11} in $L^2(\mathbb{R}^d)$. It can be viewed as a CNN which builds a representation of the data having provable properties. It has a multi-layer structure and uses convolution with wavelets $\psi_\lambda$ to propagate signals and convolution with a filter $\phi$ to generate the features (See Figure~\ref{fig:original_scattering_transform}).
\begin{figure}[h!]
	\centering
	\includegraphics[width=\textwidth]{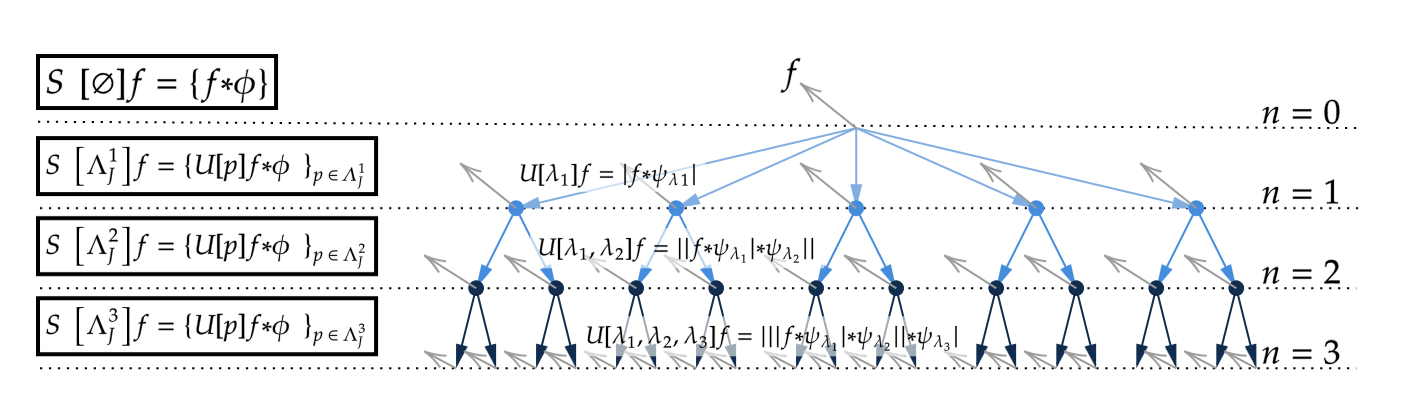}
	\caption{The structure of the scattering transform of a signal $f$.}
	\label{fig:original_scattering_transform}
\end{figure}

In this work, we introduce a scattering transform for signals in $L^2(G)$, where $G$ is a finite group, not necessarily abelian. This generalization is motivated by the need to analyze data exhibiting non-commutative symmetries, which arise naturally in various contexts. For instance, in cryptography, signals defined in the symmetric group $S_n$, representing permutations of data, inherently possess a noncommutative structure. Furthermore, the underlying symmetries of certain data transformations, such as those associated with wavelet or time-frequency analysis, can often be more accurately modeled using affine groups over finite fields. Indeed, extending the scattering transform to this setting is a natural progression, given that convolutional structures are not merely sufficient, but also necessary, for equivariance to the action of a compact group \cite{KondorTrivedi}. To achieve this, we introduce a family of wavelets on $L^2(G)$, adapted to the group structure, and use them to construct the group scattering transform. Our framework is based on representation theory and harmonic analysis, providing a rigorous foundation for the analysis of these signals. We prove that the resulting $G$-scattering transform exhibits desirable equivariance properties, respecting the underlying group symmetries, and demonstrate its stability, ensuring robustness to small perturbations in the input signal.

\subsection{Results}
Section \ref{ap:wavelet} introduces and motivates our definition of $G$-wavelets, drawing connections to existing wavelet constructions in the literature. A $G$-wavelet $\psi$ is defined as a linear combination of the characters $\{\chi^\pi: \pi \in \widehat G\}$ of the group $G$, taking into account the degree, $d_\pi$, of the unitary representations $\pi$ that give rise to the character $\chi^\pi(x) = Tr (\pi(x)),\ x\in G$. That is (see Definition \ref{Def2-2})
$$
\psi(x) = \sum_{\pi\in \widehat G} d_\pi \gamma_\psi(\pi)) \chi^\pi(x)\, \quad x\in G\,.
$$
The coefficients $\{\gamma_\psi(\pi): \pi \in \widehat G\}$ will play an important role in proving the result described below.

We define the $G$-wavelet transform of $f \in L^2(G)$ as a group convolution with a family of $G$-wavelets $\{\psi_j\}_{j = 0}^J \subset L^2(G)$, by
\[
W f (j,x) = \frac{1}{|G|} \sum_{h \in G} \psi_j(h) f(h^{-1}x) = \psi _j \ast f(x)\,.
\]
Under the Calder\'on type condition 
\begin{equation} \label{eq:Calderon-intro}
\mathscr{C}(\pi) =\sum_{j=0}^J |\gamma_{\psi_j}(\pi)|^2 = 1\,,  \quad \forall\ \pi \in \widehat G\,,
\end{equation}
the family of group translations of the $G$-wavelets form a Parseval frame,  i.e.
\begin{equation} \label{eq:parseval1}
\sum_{j = 0}^J \|W f (j,\cdot)\|_{L^2(G)}^2 = \|f\|_{L^2(G)}^2 \, , \ \forall \ f \in L^2(G)\,.
\end{equation}
In fact, conditions \eqref{eq:Calderon-intro} and \eqref{eq:parseval1} are equivalent (see Theorem \ref{teo:2.3} for a more general result).

The $G$-wavelet transform has the desirable property of being equivariant under both left and right group translations. That is, if $L_g f(x) = f(g^{-1}x)$ is the left translation operator and $R_g f(x) = f(xg)$ is the right translation operator in the group $G$, then, for all $g\in G$ and all $f\in L^2(G)$, it holds that $W(L_g f)= L_g W(f)$ and $W(R_g f)= R_g W(f)$ (see Subsection \ref{Subsec:4-3}).

Section \ref{section a} presents our notion of group scattering transform on finite groups: for a Parseval $G$-wavelet frame $\{\psi_j\}_{j = 0}^J \subset L^2(G)$, that is, one that satisfies \eqref{eq:parseval1}, and denoting by $\phi = \psi_0$, for any $m \in \mathbb{N}$ we define
\[
S[\emptyset]f = \phi \ast f \, , \ S[j_1, \dots, j_m]f = \phi \ast |\psi_{j_m} \ast \dots |\psi_{j_1} \ast f|\dots|.
\]
Here $m$ denotes the \emph{depth} of the convolutional structure. We invite the reader to see Figure \ref{fig:example} that contains a graphical description of this transform.

In Section \ref{section b}, we prove the key structural properties of this transform. Namely, with Lemma \ref{lemma4-1-New} we prove non-expansivity, that is
\[
\|Sf\|^2 := \|\phi \ast f\|_{L^2(G)}^2 + \sum_{m = 1}^\infty \sum_{j_1, \dots j_m \in \{1, \dots, J\}} \|S[j_1, \dots, j_m]f\|^2_{L^2(G)} \leq \|f\|^2_{L^2(G)} \, , \ \forall \ f \in L^2(G).
\]
This is a nontrivial property that ensures that, despite the exponential growth of the number of coefficients of the scattering transform, the total energy is controlled by the norm of the original signal. In particular, this implies that for all practical purposes, one can restrict to a finite depth of the scattering coefficients since they will become negligibly small as depth increases. We also prove, with Proposition \ref{prop Stability}, stability with respect to variations on the signals, namely we prove that
\[
\|Sf - Sf'\|^2 \leq \|f - f'\|^2_{L^2(G)} \, , \ \forall \ f, f' \in L^2(G).
\]
Note that, since $S$ is not linear, this is not a direct consequence of the non-expansivity.

A finer property that we prove, with Theorem \ref{th:energy}, is energy preservation, that is
\[
\|Sf\|^2 = \|f\|^2_{L^2(G)} \, , \ \forall \ f \in L^2(G)\,.
\]
This holds for the scattering transform of a Parseval frame wavelet if, in addition, we assume that the function  $\phi= \psi_0$ satisfies the admissibility condition $\beta := \min_{\pi\in \widehat G} |\gamma_\phi(\pi)|^2 >0$ (see Definition \ref{admissible1}). This same admissibility condition also ensures injectivity of the scattering transform. We will further discuss in Section \ref{relaxation} a relaxed admissibility condition, which relies on a structure property of the unitary irreducible representations of finite groups.

Due to the underlying equivariance property of the wavelet construction, we show with Proposition \ref{eq:equivariance2} that the group scattering transform is equivariant with respect to left and right group translation. Finally, under the previously mentioned admissibility condition, we prove in Theorem \ref{th:approxinv} that as depth increases, an approximate invariance with respect to group translations emerges. More precisely, we show that there exists a positive constant $\alpha < 1$ and a constant $C > 0$ such that
\[
\sum_{j_1, \dots j_m \in \{1, \dots, J\}} \|S[j_1, \dots, j_m] L_gf - S[j_1, \dots, j_m]f\|^2_{L^2(G)} \leq C \alpha^m \|f\|^2_{L^2(G)}
\]
for all $f \in L^2(G)$. This indicates that the scattering coefficients at increasing depths are less and less sensitive to group translations of the signal, which is a desirable property in convolutional neural networks.

Section \ref{numerical results} provides three illustrative examples of applications.  In the first one we refine the results of \cite{zou} in the context of the MNIST dataset (Section \ref{mnist}). After choosing an appropriate wavelet frame, we construct the scattering transform for abelian translations and show accuracy of more than 97\% in the classification. In the second example (Section \ref{classsound}), we demonstrate the utility of the scattering transform for the class of groups $\operatorname{Aff}(\mathbb{F}_p)$ in audio classification of signals belonging to two different categories (audio of barks and meows from dogs and cats). Here we first represent audio signals on affine groups by means of time-scale wavelet representation. This classical representation provides very poor results for classification with linear separation such as SVM. In this example, we show how to construct a group wavelet transform whose associate group scattering transform produces classification of the sound dataset with more than 87\% accuracy. In the third and last example (Section \ref{symmetricMoreover}) we consider a novel classification problem on groups: that of functions on symmetric groups. We first consider distance functions that appear in problems from cryptography, then consider random functions. In each case, we show how to use the representation theory of the symmetric group to construct a group scattering transform that is adapted to the problem.

\subsection{Related work}
The original scattering transform is defined for signals in $L^2(\mathbb{R}^d)$.
The wavelets are constructed from a mother wavelet $\psi\in L^2(\RR^d)\cap L^1(\RR^d)$, dilating and rotating it with $\lambda=2^j r$, where $j < J$ and $r$ is an element of a finite rotation group $G$, to obtain 
\begin{equation}
	\psi_{\lambda}(x)=2^{dj}\psi(2^jr^{-1}x).
\end{equation}
The convolution with the scaling function
\begin{equation}
	\phi_J(x)=2^{-dJ}\phi(2^{-J}x)
\end{equation}
localizes a function over spatial domains of size proportional to $2^J$. Such a scattering transform is a non-expansive operator and approximately invariant to translation and rotation: under strong admissibility assumptions on the wavelet and the scaling function, as the coarsest scale tends to infinity, the scattering transform is fully invariant to translations and rotations. Moreover, it is Lipschitz continuous with respect to smooth deformations. However, in dimensions greater than $1$, no wavelets has been reported satisfying such strong admissibility conditions and for the case of $L^2(\RR)$, it has been proved \cite{waldspurger2016exponential} that similar results can be achieved with much less restrictive conditions on the wavelets than the ones used in \cite{Mallat11}. The problem of the stability has been further analysed in \cite{NICOLA2023122}. There is a complementary theory to scattering of functions on $L^2(\RR^d)$ \cite{wiatowski2017mathematical} which encompasses at each layer convolution with general semi-discrete frames in the role of wavelets, general Lipschitz-continuous non-linearities in the role of the modulus, and that adds a general Lipschitz-continuous pooling operator, emulating sub-sampling and averaging.

In applied settings, classification algorithms are commonly employed in conjunction with principal component analysis (PCA) in the scattering transform. State-of-the art classification results were obtained for instance, for handwritten digit recognition and texture classification \cite{bruna2013classification}, as well as for musical genre and phone classification \cite{Anden_2014}. The method demonstrates particular promise in low-data regimes, where the number of training samples is limited, for example, it has been effectively utilized for the classification of underwater objects based on sonar signal analysis \cite{saito2017underwater}. However it has certain limitations: in \cite{cotter2017visualizing} it was shown that traditional scattering networks are not able to represent complex, edge-like patterns, which constrains their representational capacity for certain classes of visual patterns. The scattering transform has been also used to study certain types of random processes \cite{Bruna_2015}.

Beyond translations and rotations, additional classes of symmetries can, in principle, be incorporated into the scattering framework. Mallat’s original formulation of the scattering transform is theoretically flexible enough to ensure stability under the action of a compact Lie group $G$ on $L^2(\RR^d)$. However, in practice, these generalized constructions are computationally prohibitive. They rely on Littlewood–Paley decompositions on compact manifolds (see \cite{stein}) and unitary wavelet transforms on compact Lie groups \cite{Geller11}, which are challenging to implement efficiently or at scale. As a result, practical implementations tend to focus on more tractable forms of invariance. For instance, \cite{Sifre2013RotationSA} developed a scattering transform invariant to rotations, scalings, and deformations for texture discrimination, while \cite{salas2021rotation} proposed a computationally feasible architecture with built-in rotation invariance. More recently, \cite{gauthier2022parametric} introduced a learnable extension of the scattering transform based on Morlet wavelets, allowing the scales, orientations, and aspect ratios of the filters to be optimized for specific tasks.

There are a large and increasing number of papers dealing with generalizations to other domains in addition to $L^2(\RR^d)$. There are non-Euclidean extensions of the scattering transform. A geometric scattering is defined on compact, smooth, connected Riemann manifolds without boundary \cite{pmlr-v107-perlmutter20a} and on Measure Spaces \cite{chew2022geometric}, this last having properties of stability and invariance, but not energy preserving. This architecture uses wavelets similar to diffusion wavelets \cite{maggioni}. Scattering has also been generalized to graphs in \cite{zou}, where they use wavelets built in \cite{hammond} using the eigenfunctions of the Laplace operator of the graph. In Section \ref{conection} we describe the relationship between such wavelets on graphs and the $G$-wavelet introduced in the present paper.


\

\textbf{Acknowledgments:} This project has received funding from the European Union's Horizon 2020 research and innovation programme under the Marie Sklodowska-Curie Grant Agreement No 777822 and  Grants PID2019-105599GB-I00,  PID2022-142202NB-I00 / AEI
/ 10.13039/501100011033

 \section{G-wavelets}\label{ap:wavelet}
In this section we are going to design wavelets on finite groups. We will start by recalling some notations, and refer to \cite[Chapt. 15]{terras_1999} for the basic results of Fourier analysis on finite groups needed in this section.

Let $G$ be a finite group with identity $\mathbf{1}$. For $x\in G$, its conjugacy class is $[x] = \{yxy^{-1} | y \in G\}$.

A representation of $G$ is an homomorphism $\pi: G \to GL(\mathcal{H}_\pi)$, where $\mathcal{H}_\pi$ is a vector space over $\CC$ of dimension $d_\pi$. The number $d_\pi$ is called the degree of the representation $\pi$. Note that from $\pi(xy)=\pi(x)\pi(y)$ for all $x,y \in G$, one can deduce $\pi(\mathbf{1})=I_{d_\pi}$, where $I_{d_\pi}$ is the $d_\pi$-dimensional identity matrix, and $\pi(x^{-1})=\pi(x)^{-1}$ for all $x\in G$ .
The character $\chi^{\pi}$ of a representation $\pi$ of a finite group $G$ is defined by $\chi^{\pi}(x)=Tr(\pi(x))$ for all $x\in G$. The characters are complex-valued functions defined on $G$ which are constant on conjugacy classes. Indeed, let $x \in G$, let $z \in [x]$, and let $y\in G$ be such that $z=yxy^{-1}$. Then
\begin{equation}\label{eugenio:1}
 \chi^\pi(z) = Tr(\pi(yxy^{-1})) = Tr(\pi(y) \pi(x) \pi(y^{-1})) = Tr(\pi(x)) = \chi^\pi(x).
\end{equation}

Complex valued functions with this property are called class functions (see e.g. \cite[Definition on p. 258]{terras_1999}). Precisely, $f:G\to \CC$ is a class function if for all $x,y \in G$,

\begin{equation}\label{eugenio:2}
f(yxy^{-1})=f(x).
\end{equation}
A representation $\pi$ of $G$ is unitary if $\pi(x)$ is a unitary matrix for all $x\in G$.

A representation $\pi$ of $G$ is irreducible if for all $x\in G$, the linear operator $\pi(x): \mathcal{H}_\pi \to \mathcal{H}_\pi$ does not have invariant subspaces other than the trivial ones. Two representations $\pi_1$ and $\pi_2$ of $G$ are said to be equivalent if there exits an invertible square matrix $T$ such that $\pi_1(x)=T^{-1}\pi_2(x)T$ for all $x\in G$.

\begin{definition}
For a finite group $G$, the dual $\widehat{G}$ is the set of irreducible unitary representations of $G$ up to unitary equivalence.
\end{definition}

It is known (see e.g. \cite[Chapt. 15, Theorem 3]{terras_1999}) that the number of elements of $\widehat{G}$ coincides with the number of conjugacy classes of $G$. Let us denote by $k$ this number, let us choose a system of inequivalent unitary irreducible representation so that we can identify $\widehat{G}=\{\pi^1, \ldots, \pi^k\}$, and for each $i=1,2, \ldots, k$, write $\chi^i$ for the character of $\pi^i.$ Denote by $\Cl_1, \ldots, \Cl_k$ the conjugacy classes of $G$, where $\Cl_1=[\mathbf{1}]=\{\mathbf{1}\}$. By \eqref{eugenio:1}, each $\chi^i$ is constant in each class $\Cl_r$, $1\leq r \leq k$, and we denote by $\chi^i_r$ the value of $\chi^i$ in the conjugacy class $\Cl_r$. The square matrix of size $k\times k$ formed with the elements $\chi^i_r$ is called the character table of $G$. Observe that if $d_i$ is the degree of $\pi^i \in \widehat{G}$, then $d_i=\chi^i_{1}$. There exist algorithms to calculate the character table of finite groups (see \cite{grove_1997}(Chapter 7), \cite{Burnside1911}, \cite{dixon_1967}, \cite{dixon_1970} and \cite{schneider_1990}).

A relevant unitary representation of $G$ is the left regular representation $L$. It is defined on the Hilbert space $L^2(G) = \{f:G \to \CC\}$, which is considered to be endowed with the inner product
\begin{equation}\label{eq:scalarprod}
    \langle f, g \rangle = \frac{1}{|G|} \sum_{x\in G} f(x)\overline{g(x)}, \quad f,g\in L^2(G),
\end{equation}
and the induced norm $\|f\|=(\langle f, f \rangle)^{1/2}$, by
$$
L(h)f(x) = f(h^{-1}x) \, ,  \quad \textnormal{for } f \in L^2(G) \textnormal{ and } x,h \in G.
$$
Recall that $L$ is equivalent to the direct sum of irreducibles (see e.g. \cite[Chapt. 15, Lemma 2]{terras_1999})
\begin{equation}\label{eq:directsum}
L(x) \approx \bigoplus_{r = 1}^k d_r \pi^r(x) \quad \forall \, x \in G
\end{equation}
where $d_r \pi^r$ indicates the direct product of $d_r$ copies of the representation $\pi^r$.

The convolution of $f, g \in L^2(G)$ is defined as
\begin{equation}\label{def:convolution}
f*g(x) =\frac{1}{|G|}\sum_{y\in G} f(y) L(y) g(x) = \frac{1}{|G|}\sum_{y\in G} f(y)g(y^{-1}x) = \frac{1}{|G|} \sum_{z\in G} f(xz^{-1})g(z), \quad f,g\in L^2(G).
\end{equation}

The convolution is in general noncommutative. But if $f\in L^2(G)$ is a class function (see \eqref{eugenio:2}) then, for all $g \in L^2(G)$ and all $x \in G$
\begin{equation}\label{eugenio:5}
    f*g(x) = \frac{1}{|G|} \sum_{z\in G} f(xz^{-1})g(z) = \frac{1}{|G|}\sum_{z\in G} g(z)f(x(z^{-1}x)x^{-1}) = \frac{1}{|G|}\sum_{z\in G} g(z)f(z^{-1}x) = g*f(x).
\end{equation}

The Fourier transform of $f\in L^2(G)$ is defined as the matrix valued function (see Chapters 3 and 5 of \cite{folland_1995} and Chapter 15 of \cite{terras_1999})
\[
\widehat{f}(\pi^r)=\frac{1}{|G|}\sum_{x\in G} f(x) \pi^r(x^{-1}), \quad \pi^r\in\widehat{G}.
\]
It factorizes convolutions in the following way: for $f, g \in L^2(G)$, it is easy to see that
\begin{equation}\label{eq:fourierfactor}
\widehat{f\ast g}(\pi^r) = \widehat{g}(\pi^r) \widehat{f}(\pi^r).
\end{equation}
Moreover, for finite groups one can write Plancherel's theorem as (see \cite[Chapt. 15, Theorem 2]{terras_1999}, but observe that here we are considering the normalized scalar product)
\begin{equation}\label{eq:Plancherel}
\|f\|^2_{L^2(G)} = \sum_{r = 1}^k\sum_{i,j=1}^{d_r} d_r \left|\left\langle f, \pi^r_{ij}\right\rangle\right|^{2} = \sum_{r = 1}^k\sum_{i,j=1}^{d_r} d_r \left|\widehat{f}(\pi^r)_{ij}\right|^{2}
\end{equation}
for any $f \in L^2(G)$, where $\pi^r_{ij} : G \to \CC$ is the $ij$-th element of the representation $\pi^r$.

\subsection{G-Wavelet transform} \label{G-wavelet transform}

The definition of the $G$-wavelets is inspired by the idea of \cite{hammond} (see Section \ref{conection}).
\begin{definition} \label{Def2-2}
Given a measure space $(\Omega,\mu)$, we call \emph{kernel} a family of functions $\gamma = \{\gamma_\omega: \{1,\ldots,k\} \to \CC\}_{\omega \in \Omega}$. The \textit{$G$-wavelet} associated to the kernel $\gamma_\omega$ is
\begin{equation}\label{def:2.1}
\psi_{\gamma_\omega}(x):=\sum_{r=1}^k d_r \gamma_\omega(r)\chi^r(x), \quad x\in G.
\end{equation} 
\end{definition}

Assume that for $\omega\in \Omega$, we have $\displaystyle B_\omega := \max_{1\leq r \leq k} |\gamma_\omega(r)| < \infty$. Then $\psi_{\gamma_\omega}
\in L^\infty(G)$. Indeed, since the trace of a unitary matrix is less than or equal to the dimension of the matrix,
\begin{equation} \label{eq:infinity}
	\|\psi_{\gamma_\omega}\|_\infty = \max_{x\in G} |\psi_{\gamma_\omega}(x)| \leq \sup_{x\in G} \sum_{r=1}^k d_r |\gamma_\omega(r)| |\chi^r(x)| \leq B_\omega \sum_{r=1}^k d_r^2 = B_\omega |G| < \infty\,, 
\end{equation}
where the have used that $\displaystyle \sum_{r=1}^k d_r^2 = |G|$ (see \cite[Chapt. 15, Lemma 2]{terras_1999}).

The $G$-wavelets given in \eqref{def:2.1} are class functions. This can be seen directly, using the property of the trace:
\begin{align*}
\psi_{\gamma_\omega}(h^{-1} x h) & = \sum_{r=1}^k d_r \gamma_\omega(r)\chi^r(h^{-1}x h) = \sum_{r=1}^k d_r \gamma_\omega(r) Tr(\pi^r(h^{-1})\pi^r(x)\pi^r(h))\\
& = \sum_{r=1}^k d_r \gamma_\omega(r) Tr(\pi^r(x)) = \sum_{r=1}^k d_r \gamma_\omega(r)\chi^r(x).
\end{align*}
We also have that the Fourier transform of a $G$-wavelet is a multiple, given by the defining kernel, of the identity operator on $L^2(G)$:
\begin{equation}\label{eq:FourierGwavelet}
\widehat{\psi_{\gamma_\omega}}(\pi^r) = \gamma_\omega(r) I_{d_r}.
\end{equation}
This can be obtained by direct computation of the matrix elements starting from \eqref{def:2.1}:
\begin{align*}
\widehat{\psi_{\gamma_\omega}}(\pi^r)_{i,j} & = \frac{1}{|G|}\sum_{x\in G}  \sum_{s=1}^k d_s \gamma_\omega(s)\chi^s(x) \pi^r(x^{-1})_{i,j}
= \frac{1}{|G|}  \sum_{s=1}^k d_s \gamma_\omega(s) \sum_{x\in G} Tr(\pi^s(x)) \pi^r(x^{-1})_{i,j}\\
& = \frac{1}{|G|}  \sum_{s=1}^k d_s \gamma_\omega(s) \sum_{l = 1}^{d_s} \sum_{x\in G} \pi^s(x)_{l,l} \pi^r(x^{-1})_{i,j} = \gamma_\omega(r) \delta_{i,j}
\end{align*}
where the last identity is Schur's orthogonality theorem (see e.g. \cite[Chapt. 15, Theorem 1]{terras_1999}).

We define the $G$-wavelet transform in terms of the left translates of a $G$-wavelet, as follows.
\begin{definition} \label{Wavelet-transform}
Given a kernel $\gamma$, we define the \textit{$G$-wavelet  transform} of $f\in L^2(G)$ by 
	\[
	W_{\gamma}f(\omega,h) :=  f\ast \psi_{\gamma_\omega}(h) = \frac{1}{|G|} \sum_{x \in G} f(x) \psi_{\gamma_\omega}(x^{-1}h), \quad \omega \in \Omega, h\in G.
	\]
\end{definition}

Since the wavelets are class functions, by \eqref{eugenio:5} we can write \begin{equation}\label{eq:Wconv}
		W_{\gamma}f(\omega,h) = f \ast \psi_{\gamma_\omega} (h) = \psi_{\gamma_\omega}  \ast f (h).
	\end{equation}

Assume that for $\omega\in \Omega$, we have $\displaystyle B_\omega := \max_{1\leq r \leq k} |\gamma_\omega(r)| < \infty$. By \eqref{eq:infinity} and the fact that $G$ is finite,
$\psi_{\gamma_\omega}\in L^\infty(G) \subset \ L^1(G).$ Hence, by Young's inequality and \eqref{eq:Wconv} we conclude that $W_{\gamma}f(\omega,\cdot)\in L^2(G).$
	
Moreover, using \eqref{eq:fourierfactor}, we have that 
	\[
	\widehat{W_{\gamma}f}(\omega,\cdot)(\pi^r) = \widehat{f}(\pi^r) \widehat{\psi_{\gamma_\omega}}(\pi^r).
	\]
	Using \eqref{eq:FourierGwavelet}, we immediately obtain the following result.
\begin{lemma}\label{lem:FourierW}
	Given a kernel $\gamma$ as in Definition \ref{Def2-2}, the \textit{$G$-wavelet transform} of $f\in L^2(G)$ satisfies
	\[
	\widehat{W_{\gamma}f}(\omega,\cdot)(\pi^r) = \gamma_\omega(r) \, \widehat{f}(\pi^r).
	\]
\end{lemma}


\subsection{Properties} \label{Properties}
Let us first show that, if a Calder\'on type condition holds, we have a reconstruction formula for the introduced $G$-wavelet transform, analogous to that for the continuous wavelet transform.

In the sequel, for $f\in L^2(G)$ we write $f^\dagger(x) = \overline{f(x^{-1})}$ for the natural involution in $L^2(G)$. As in \eqref{eq:FourierGwavelet}, it can be proved that
	\begin{equation}\label{eq:FourierGwavelet-dagger}
		\widehat{\psi_{\gamma_\omega}^\dagger}(\pi^r) = \overline{\gamma_\omega(r)} I_{d_r}.
	\end{equation}

\begin{theorem}[Inverse formula]
If the kernel $\gamma = \{\gamma_\omega:\{1,\ldots, k\} \to \CC\}_{\omega\in \Omega}$ satisfies the condition
\begin{equation}\label{calderon}
\int_\Omega |\gamma_\omega(r)|^2 d\mu(\omega) = 1 \quad \forall \ r\in\{1,\ldots, k\},
\end{equation}
then for all $f\in L^2(G)$
\begin{equation}
f(x) = \frac{1}{|G|}\sum_{h\in G}\int_\Omega W_{\gamma}f(\omega,h)\, L(h)\psi_{\gamma_\omega}^\dagger(x) \, d\mu(\omega),
\end{equation}
for every $x\in G$.
\end{theorem}

\begin{proof}
By Definition \ref{Wavelet-transform} we can write
\begin{align*}
	\frac{1}{|G|}\sum_{h\in G}\int_\Omega W_{\gamma_\omega} f(h) \, L(h)\psi_{\gamma_\omega}^\dagger(x) \, d\mu(\omega) & = \frac{1}{|G|}\sum_{h\in G}\int_\Omega f\ast\psi_{\gamma_\omega}(h) \psi_{\gamma_\omega}^\dag (h^{-1}x)d\mu(\omega)\\
	& =\int_\Omega f\ast\psi_{\gamma_\omega}\ast\psi_{\gamma_\omega}^\dag(x)d\mu(\omega).
\end{align*}

By \eqref{eq:fourierfactor} and \eqref{eq:FourierGwavelet-dagger} we have that $\widehat{\psi_{\gamma_\omega}\ast\psi_{\gamma_\omega}^\dag}(\pi^r) = |\gamma_\omega(r)|^2 I_{d_r}$, so that, by the same argument leading from \eqref{def:2.1} to \eqref{eq:FourierGwavelet}, the second convolution reads
\[
\psi_{\gamma_\omega}\ast\psi_{\gamma_\omega}^\dag(x) = \sum_{r=1}^kd_r |\gamma_\omega(r)|^2 \chi^r(x) .
\]
Thus, using Calder\'on's condition \eqref{calderon}, and recalling \eqref{eq:directsum}, we get
\begin{align*}
\int_\Omega f\ast\psi_{\gamma_\omega}\ast\psi_{\gamma_\omega}^\dag(x)d\mu(\omega) & = \sum_{r=1}^k d_r \int_\Omega |\gamma_\omega(r)|^2d\mu(\omega) f*  \chi^r(x) = f \ast \left(\sum_{r=1}^k d_r \chi^r \right)(x)\\
& = f \ast Tr\left(\sum_{r=1}^k d_r \pi^r \right)(x) = f \ast Tr (L) (x) = f(x)
\end{align*}
where the last identity is due to the fact that the trace of the left regular representation is $|G|$ when computed in the identity, and zero otherwise.
\end{proof}

For any practical computation, $\Omega$ must be a finite number of scales $\Omega = \{0,\dots,J\}$. The natural question is how appropriated the $G$-wavelet set $\{\psi_{\gamma_j}| 0 \leq j \leq J\}$ will be for representing functions on the group. This question can be addressed by characterizing when the system of translates is a frame of $L^2(G)$, which is the content of the next theorem.
%
\begin{theorem}\label{teo:2.3}
Let $\{\gamma_j:\{1,\ldots, k\} \to \CC\}_{0 \leq j \leq J}$ and denote by
\[
\mathscr{C}(r) = \sum_{j=0}^{J}  |\gamma_j(r)|^2 \, , \ r \in \{1,\dots,k\}.
\]
The following are equivalent
\begin{itemize}
\item[i)] The set $\{L(h)\psi_{\gamma_j} \,|\, 0\leq j \leq J, h\in G\}$ is a frame of $L^2(G)$ with constants $0<A\leq B < \infty$, i.e.
\begin{equation}\label{eq:14}
A\|f\|^2\leq \frac{1}{|G|}\sum_{h\in G} \sum_{j=0}^{J} |\langle f, L(h)\psi_{\gamma_j} \rangle|^2\leq B\|f\|^2 \quad \forall \  f\in L^2(G).
\end{equation}
\item[ii)] There exist two constants $0 < A \leq B < \infty$ such that
\begin{equation}\label{eq:Calderon}
A \leq \mathscr{C}(r) \leq B \quad \forall \ r  \in \{1,\dots,k\}.
\end{equation}
\end{itemize}
\end{theorem}
\begin{proof}
For all $h\in G$ and all $j\in\{0,1, \dots, J\}$ we have
\[
\langle f, L(h)\psi_{\gamma_j}\rangle = \frac{1}{|G|}\sum_{x\in G}  f(x) \overline{\psi_{\gamma_j}(h^{-1}x)} = \frac{1}{|G|}\sum_{x\in G}  f(x) \psi_{\gamma_j}^\dag(x^{-1}h) = f\ast \psi_{\gamma_j}^\dag(h)\,.
\]
By \eqref{eq:fourierfactor} and \eqref{eq:FourierGwavelet-dagger}, for any $\pi^r\in \widehat G$,
\[
\widehat {f\ast \psi_{\gamma_j}^\dag}(\pi^r) = \overline{\gamma_j(r)} \widehat f(\pi^r)\,.
\]

By Plancherel's identity \eqref{eq:Plancherel}, for all $j \in \{0,\dots,J\}$ we have
\[
\frac{1}{|G|}\sum_{h\in G} |\langle f, L(h)\psi_{\gamma_j}\rangle|^2 = \|f\ast \psi_{\gamma_j}^\dag\|_{L^2(G)}^2 = \sum_{r=1}^k\sum_{m,n=1}^{d_r}d_r |\gamma_j(r)|^2 |\widehat{f}(\pi^r)_{nm}|^2.
\]
Then
\begin{equation}\label{eq:CalderonFourier}
\frac{1}{|G|}\sum_{h\in G} \sum_{j=0}^{J} |\langle f, L(h)\psi_{\gamma_j}\rangle|^2
= \sum_{r=1}^k \left(  \sum_{j=0}^{J}  |\gamma_j(r)|^2\right)d_r \sum_{m,n=1}^{d_r} |\widehat{f}(\pi^r)_{nm}|^2 .
\end{equation}
We can now complete the proof by recalling that  Plancherel's identity \eqref{eq:Plancherel} reads
\begin{equation*}
\sum_{r=1}^k\sum_{m,n=1}^{d_r}d_r|(\widehat{f}(\pi^r))_{nm}|^2=\|f\|^2.
\end{equation*}
Indeed, if \eqref{eq:Calderon} holds, then one immediately obtains the frame condition \eqref{eq:14}. On the other hand, suppose that \eqref{eq:14} holds and suppose, by contradiction, that \eqref{eq:Calderon} is not true because $\mathscr{C}(r_0) < A$ for one $r_0 \in \{1, \dots, k\}$. Then, if we choose $f_0 \in L^2(G)$ to be such that $\widehat{f_0}(\pi^r) = \delta_{r,r_0}M$ for a nonzero $d_{r_0} \times d_{r_0}$ matrix $M$, we get that \eqref{eq:CalderonFourier} reads
\[
\frac{1}{|G|}\sum_{h\in G} |\langle f, L(h)\psi_{\gamma_j}\rangle|^2 = d_{r_0} \mathscr{C}(r_0) \sum_{m,n=1}^{d_r} |\widehat{f}(\pi^r)_{nm}|^2 < A d_{r_0} \sum_{m,n=1}^{d_r} |\widehat{f}(\pi^r)_{nm}|^2 = A \|f\|^2
\]
which contradicts \eqref{eq:14}. The other inequality can be treated similarly.
\end{proof}

\begin{remark}\label{remark2.7}
	The same result as in Theorem \ref{teo:2.3} holds if we consider the set $\{L(h)\psi_{\gamma_j}^\dag \,|\, 0\leq j \leq J, h\in G\}.$ In fact, 
	$
	\langle f, L(h)\psi_{\gamma_j}^\dag\rangle  = f\ast \psi_{\gamma_j}(h)\,.
	$
	and by \eqref{eq:fourierfactor} and \eqref{eq:FourierGwavelet}, for any $\pi^r\in \widehat G$,
	$
	\widehat {f\ast \psi_{\gamma_j}}(\pi^r) = \gamma_j(r) \widehat f(\pi^r)\,
	$
	so that the above proof goes unchanged to obtain
	$$
	A\|f\|^2\leq \frac{1}{|G|}\sum_{h\in G} \sum_{j=0}^{J} |\langle f, L(h)\psi_{\gamma_j}^\dag \rangle|^2 = \sum_{j=0}^J \|\psi_{\gamma_j}\ast f\|^2 \leq B\|f\|^2 \quad \forall \  f\in L^2(G).
	$$
\end{remark}

\subsection{Connection with Hammond et al. Wavelets}\label{conection}

In \cite{hammond} D.K. Hammond, P. Vandergheynst, and R. Givonval define and study wavelets on weighted connected finite graphs. Let $V$ be the set of vertices of the graph, denoted by $V=\{1,2, \dots N\}$, $E$  the set of edges of the graph and $w:E \to \mathbb R^+$ the weights. 

The adjacency matrix of the weighted graph is the $N\times N$ matrix $A=(a_{m,n})$ where 
\begin{equation*}
	a_{m,n}=\begin{cases}
		w(e), \text{ if } e\in E \ \text{connect vertices} \  m \  \text{and}\  n\\
		0, \text{ otherwise}.
	\end{cases}
\end{equation*}
For each vertex $m$ in $V$, denote by $d(m)$ the sum of the weights of all edges incident to $m$. Let $D$ be the matrix with diagonal entries $d(1), d(2), \dots, d(m)$, and 0 otherwise. The Laplacian of the weighted graph $\{V, E, w\}$ is the $N\times N$ matrix
\[
\mathcal L = D - A.
\]
The matrix $\mathcal L$ is real and symmetric. Therefore, it has a complete set of orthonormal eigenvectors, denoted by $\chi_\ell, \ell = 1, 2, \dots. N$, with associated eigenvalues $\lambda_l$, that is $\mathcal L \chi_\ell = \lambda_l \chi_\ell$. Moreover, $\lambda_1=0$ is an eigenvalue of multiplicity one and the remainder eigenvalues are all positive. (See \cite{hammond} for details.)

For a kenel $g: \mathbb R^+ \to \mathbb R^+$ the \textit{spectral graph wavelets} with parameter $t\in \mathbb R^+$ and $n=1, 2, \dots, N$ are given in \cite{hammond} by
\begin{equation} \label{Eq:SGW}
	\psi_{t,n} (m) = \sum_{\ell=1}^N g(t\lambda_\ell) \chi_{\ell}(n) \chi_{\ell} (m)
\end{equation}

The purpose of this section is to show that if $G$ is a finite abelian group, the $G$-wavelets defined in \eqref{def:2.1} coincide with the spectral graph wavelets defined in \eqref{Eq:SGW} for appropriate kernels $\gamma_\omega$ and $g$. Observe that our results in Sections \eqref{G-wavelet transform} and \eqref{Properties} are valid for finite groups not necessarily abelian.

The connection between our wavelets and the spectral graph wavelets of  \cite{hammond} is provided by the Cayley graph of a group.

\begin{definition}
	Given a group $G$ and a symmetric set of generators $S$ of $G$ (meaning $S$ generates $G$ and $S$ is closed under taking inverses) , the Cayley graph $X =Cay(G, S)$ of $G$ is a graph whose vertices are the elements of $G$,  two elements $x$ and $y$ in $G$,
	are connected if there exists an element $s\in S$ such that
	$x = sy$ and all the weights of the edges are 1.
\end{definition}

For $G$ abelian group with $N$ elements, all unitary irreducible representations are unidimensional and they are equal to the characters of the group. We write
\[
\widehat G =\{ \chi^1, \chi^2, \dots, \chi^N \,. \}
\]

There is a connection between the characters of the abelian group and the eigenvalues and eigenvectors  of the Laplacian of the Cayley graph of the group, as the following theorem (due to Lovász \cite{Cayley}), shows.
\begin{theorem} Let $G$ be an abelian group of order $N$ with
	 characters $\chi^1, \chi^2, \ldots, \chi^N$. Let $S\subset G$ be a symmetric set of generators of $G$. Then the eigenvalues of the Laplacian of the  Cayley graph $X=Cay(G,S)$ of $G$ are given by
	\begin{equation}
		\lambda_i=|S|-\sum_{x\in S}\chi^i(x)
	\end{equation}
	Moreover, an eigenvector corresponding to $\lambda_i$ is $(\chi^i(x) : x \in G)^t$ (that is, the column vector whose $x$-entry is $\chi^i(x)$).
\end{theorem}

Given a kernel $g: \mathbb R^+ \to \mathbb R^+$, let $\gamma_t : \{1,2, \dots, N\} \to \mathbb R^+$ be given by $\gamma_t (\ell) = g(t\lambda_l)$ for $t\in \mathbb R^+.$ Then, for $x\in G$, equation \eqref{def:2.1} reads:
\begin{equation} \label{eq:connection}
	\psi_{\gamma_t}(x)=\sum_{\ell=1}^N g(t\lambda_\ell)\chi^l(x),
\end{equation}
which coincides with \eqref{Eq:SGW} when $n$ is the identity of the group. For any other $h\in G$, the wavelets
\[
\psi_{\gamma_t, h}(x) = \sum_{\ell=1}^N g(t\lambda_\ell) \chi^{\ell}(h) \chi^{\ell} (x)= \sum_{\ell=1}^N g(t\lambda_\ell) \chi^{\ell}(hx) = \psi_{\gamma_t}(hx)
\]
are all translations of the wavelet given by \eqref{eq:connection}.

\begin{example}
	Let $p$  be a prime number and let $G=(\ZZ_p^*, *)$ be the multiplicative group of integers mod ($p$) minus $0$. Let $x$ be a generator of $(\ZZ_p^*, *)$  and $\omega = e^{2\pi i /(p-1)}$. The character table of $G$ is
	$$
	\begin{array}{r|rrrrr} 
		 & 1 & x & x^2 & \cdots & x^{p-2} \\
		\hline \chi^0 & 1 & 1 & 1 &\cdots & 1 \\
		\hline \chi^1 & 1 & \omega & \omega^2 & \cdots & \omega^{p-2}\\
		\hline \chi^2 & 1 & \omega^2 & \omega^4 & \cdots & \omega^{p-3} \\
		\hline \vdots & \vdots & \vdots & \vdots & \cdots & \vdots\\
		\hline \chi^{p-2} & 1 & \omega^{p-2} & \omega^{p-3} & \cdots & \omega
	\end{array}
	$$
    The Cayley graph of $G$ with $S=\{x, x^{-1}\}$ is $C_{p-1}$, a cycle graph with $p-1$ vertexes. 
    \begin{figure}[H]
        \centering
        \includegraphics[width=0.3\linewidth]{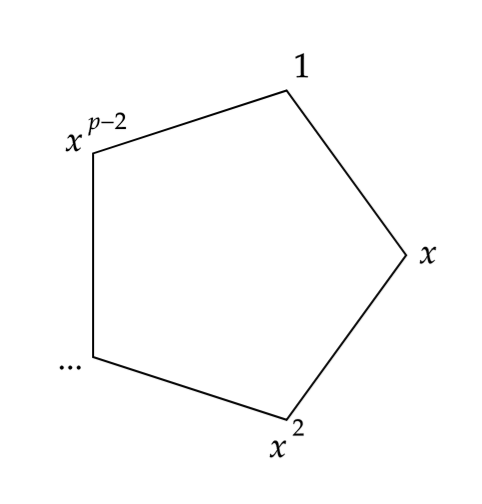} \
        \caption{Representation of the Cayley graph of $(\ZZ_p^*, *)$}
        \label{fig:cayCp}
    \end{figure}
    The adjacency matrix of $C_{p-1}$ is
    \begin{equation*}
        \begin{pmatrix}
            0 & 1 & 0 & \ldots & 0 & 1\\
            1 & 0 & 1 & \ldots & 0 & 0\\
            0 & 1 & 0 & \ldots & 0 & 0\\
            \vdots&\vdots&\vdots& \ddots & \vdots&\vdots\\
            0 & 0 & 0 & \ldots & 0 & 1\\
            1 & 0 & 0 & \ldots & 1 & 0
            \end{pmatrix}
    \end{equation*}
    It is a circulant matrix and their eigenvectors are $(\chi^i(x) : x \in \ZZ_p^*)$. The wavelets associated to a kernel $\gamma_z: \{0,1, p-2\} \to \mathbb C, z\in \Omega$  are: 
\begin{equation}
    	\psi_{\gamma_z}(x^m)=\sum_{n=0}^{p-2} \gamma_z (n)\omega^{nm}.
    \end{equation}
\end{example}

\section{Construction of Scattering Networks on Finite Groups}\label{section a}

Our construction will rely on the $G$-wavelets defined in the previous section. Consider as measure space the finite set $\Omega = \{0,\dots,J\}$ and, for a finite group $G$ whose dual has size $k$, consider a \textit{kernel} $\{\gamma_j: \{1,\ldots, k\} \to \CC\}_{0\leq j\leq J}$. The $k(J+1)$ values of this kernel are the only parameters of the network to be defined. Let us also denote by $\{\psi_j = \psi_{\gamma_j}\}_{0\leq j\leq J}$ the associated family of $G$-wavelets defined in \eqref{def:2.1}. For more clarity in the presentation of the scattering network construction, we will denote by $\phi = \psi_0$, and set the kernel in such a way that
\begin{equation}\label{multir}
\sum_{j = 1}^J |\gamma_j(r)|^2 = 1 - |\gamma_0(r)|^2 \,, \quad r=1,2, \dots, k,
\end{equation}
hence constraining the $G$-wavelet to be a Parseval frame  (see Theorem \ref{teo:2.3}).

\

For a fixed $j\in \{1,..,J\}$, we call \textit{one-step $G$-propagator} $U[j]:L^2(G) \to L^2(G)$ the operator defined by
\begin{equation}
U[j]f:=|\psi_j*f|, \quad f \in L^2(G).
\end{equation}

For $m, J > 0$, we denote by $\Lambda_J^m=\{1, \ldots, J\}^m$ and call it the set of $J$-paths of length $m$. For $m > 0$ and $p = (j_1, \dots, j_m) \in \Lambda_J^m$, the \textit{$m$-steps $G$-scattering propagator} $U[p]:L^2(G) \to L^2(G)$ is the operator defined by
\begin{equation}
    U[p]f:=U[j_m]\ldots U[j_1]f, \quad f \in L^2(G).
\end{equation}
Observe that $U[p]$ takes $L^2(G)$ into $L^2(G)$ due to the fact that $\psi_j\in L^\infty(G) \subset L^1(G), j=1,2, \dots, J$ (see \eqref{eq:infinity}).

The collection of all $J$-paths of finite length will be denoted by $P_J:=\bigcup_{m=0}^\infty \Lambda_J^m$, with the convention that $\Lambda_J^0 = \emptyset$. With a slight abuse of notation, we will also indicate, for the empty path, $U[\emptyset]f=f$.

For a path $p=(j_1, \ldots, j_m) \in \Lambda_J^m$ and a $j_{m+1}\in\{1,\ldots,J\} = \Lambda_J^1$, following \cite{Mallat11} we define the path $p +j_{m+1}=(j_1, \ldots,j_m, j_{m+1}) \in \Lambda_J^{m+1}$. Then
\begin{equation}
    U[p + j_{m+1}] := U[j_{m+1}]U[p].    
\end{equation}
We will use the function $\phi$ introduced above to define the \textit{$G$-average operator} $Q:L^2(G)\to L^2(G)$ as
\begin{equation}
Q f= \phi*f \, , \quad f\in L^2(G).
\end{equation}
In the nomenclature introduced in \cite{Mallat11} we will then call \textit{$G$-windowed scattering operator} for a path $p\in P_J$ the map $S[p]:L^2(G)\to L^2(G)$ defined by
\begin{equation}
S[p]f:= Q U[p] f, \quad f\in L^2(G).
\end{equation}
Observe that $S[p]$ takes $L^2(G)$ into $L^2(G)$ due to the fact that $\phi\in L^\infty(G) \subset L^1(G) $ (see \eqref{eq:infinity}).
\begin{definition}
The \textit{$G$-scattering transform} with $J$ generators is defined as
\[
S[P_J]f = \{S[p]f\}_{p\in P_J}, \quad f \in L^2(G).
\]
We will denote by $\|\cdot\|_{2,J}$ its $\ell_2(P_J,L^2(G))$ norm, given by
\begin{align*}
\|S[P_J]f\|_{2,J} =\left(\sum_{p\in P_J} \|S[p]f\|^2\right)^{1/2} =\left(\sum_{m=0}^\infty \sum_{(j_1,\ldots, j_m) \in \Lambda_J^m} \|Q U[j_m]\ldots U[j_1]f\|^2\right)^{1/2}.
\end{align*}
\end{definition}
The $G$-scattering transform acts as a group convolutional neural network. At the $m$-th layer, the \emph{propagated signal} is $\{U[p]f:p \in\Lambda_J^m\}$ and the \emph{extracted features} are $\{S[p]f:p \in\Lambda_J^{m-1}\}$.
\begin{figure}[ht] 
		\centering
		\includegraphics[width=\textwidth]{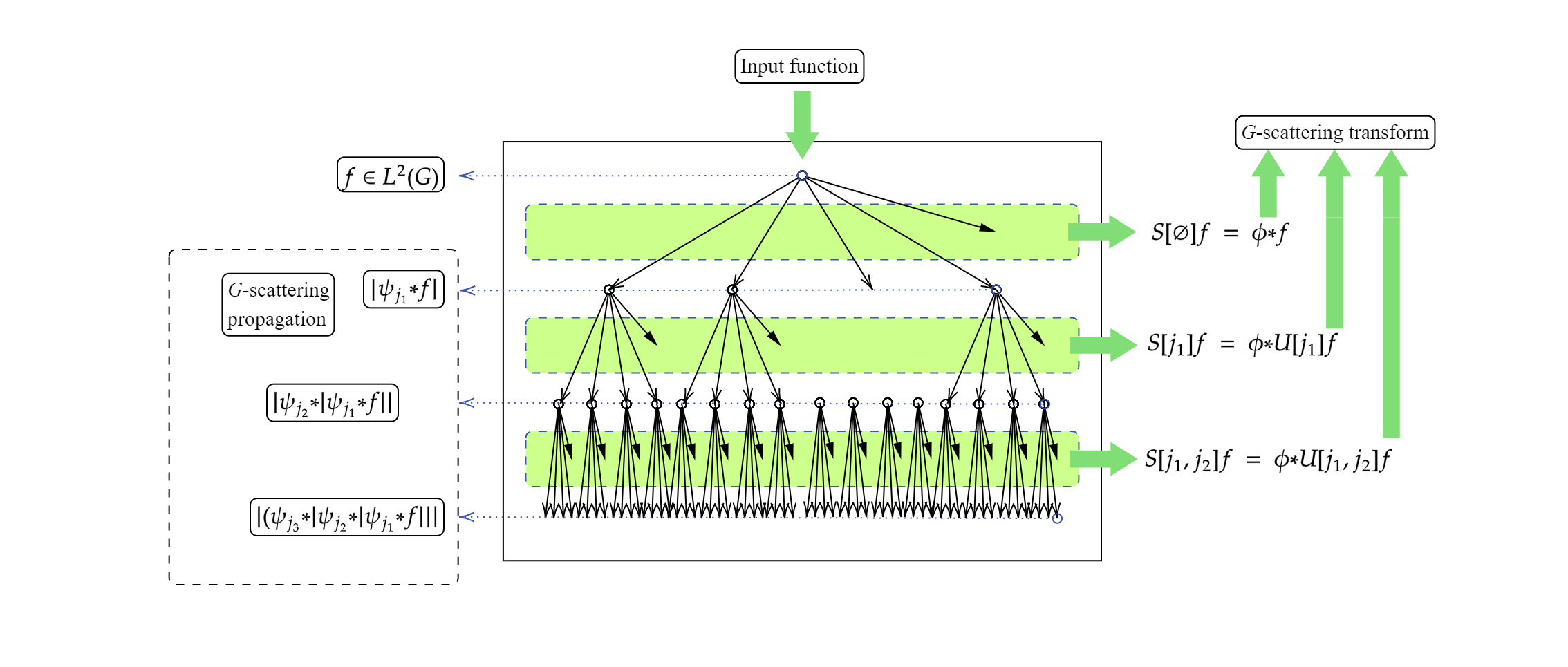}
		\caption{G-Scattering transform}
		\label{fig:example}
	\end{figure}

\begin{remark}
The scattering networks constructed in this section follows the original idea of S. Mallat in \cite{Mallat11} (see, also, Figure 1 in the Introduction). For this scattering transform to have good properties, the wavelets used must satisfy the equations that appear in Proposition 2.7 of \cite{Mallat11}.

The same idea is used by D. Zou and G. Lerman in \cite{zou} to construct scattering networks in weighted connected finite graphs. To prove good properties for their scattering transform, they use a scaling function $\phi$ and a wavelet $\psi$ that satisfy
	$$
	|\widehat{\phi}(2^J \omega)|^2 + \sum_{j>-J} |\widehat{\psi} (2^{-j}\omega)|^2 = 1\,, \quad \omega\in \mathbb R\,,
	$$
that can be achieved with the theory of Multiresolution Analysis (MRA).

While our $G$-scattering transform is also a cascade of modulus of convolutions, our wavelets are different from the ones in \cite{Mallat11} and \cite{zou}. In these works the scattering network depends only on the chosen wavelet and at each level $m\geq 1$ there are infinitely many paths selected from $\{(j_1, \dots, j_m): j_i \geq -J\}$. In our case, described above, our wavelets in \eqref{def:2.1} depend on $(J+1)k$ parameters $\{\gamma_j(r): 1\leq r\leq k\}_{0\leq j \leq J}$, where $k$ is the number of conjugacy classes of $G$, with the restriction \eqref{multir}. Therefore, G-wavelets are easy to construct and at each level $m\geq 1$ the number of features extracted is finite.
\end{remark}

Now we introduce a structural property of the scattering transform that will be needed in Chapter \ref{section b}, which states that, in the scattering networks, the propagated norm at the $m$-th layer splits into the propagated norm at the next layer plus the output energy at the current layer.

\begin{lemma}\label{lem:3.2.2}  Let $\{\gamma_j: \{1,\ldots, k\} \to  \CC\}_{0\leq j\leq J}$ be a kernel satisfying \eqref{multir}, and let $U[\cdot]$ and $S[\cdot]$ be the associated $G$-scattering propagator and the $G$-windowed scattering operator. Then for all $m \geq 0$
	\begin{equation}\label{eq:36}
		\sum_{p\in \Lambda_J^m}\|U[p]f\|^2 = \sum_{p\in \Lambda_J^{m+1}}\|U[p]f\|^2+\sum_{p\in \Lambda_J^{m}}\|S[p]f\|^2.
	\end{equation}
\end{lemma}
\begin{proof}
The case $m=0$ is a consequence of condition \eqref{multir} and  Remark \ref{remark2.7} with $A = B = 1$. Indeed,
\[
\|f\|^2 = \sum_{j=0}^{J}\|\psi_j *f\|^2 = \|\phi \ast f\|^2 + \sum_{j=1}^{J}\|\psi_j *f\|^2 = \|Q f\|^2 + \sum_{j=1}^{J}\|U[j]f\|^2.
\]
By repeating this argument replacing $f$ with $U[p]f$ we can obtain, for  any $p\in \Lambda_J^m$
\[
\|U[p]f\|^2 = \sum_{j=1}^{J}\|U[p+j]f\|^2 +\|S[p]f\|^2.
\]
Finally, we can conclude by summing the previous equation for $p\in \Lambda_J^m$:
\[
\sum_{p\in \Lambda_J^m}\|U[p]f\|^2 = \sum_{p\in \Lambda_J^m}\sum_{j=1}^{J}\|U[p+j]f\|^2 + \sum_{p\in \Lambda_J^m}\|S[p]f\|^2 = \sum_{p\in \Lambda_J^{m+1}}\|U[p]f\|^2 + \sum_{p\in \Lambda_J^m}\|S[p]f\|^2
\]
where the last identity is due to
\[
\Lambda_J^{m+1}=\bigcup_{p\in \Lambda_J^m}\bigcup_{j=1}^{J} p+j. \qedhere
\]
\end{proof}
\section{Properties of the $G$-scattering transform} \label{section b}
The number of paths over which scattering can be computed is theoretically infinite; however, under condition \eqref{multir}, the energy of the $G$-scattering transform is not expanding, as it is made precise in the following Lemma. This result is a consequence of the structural property of the scattering transform stated in Lemma \ref{lem:3.2.2}\,.
\begin{lemma} \label{lemma4-1-New} ({\bf Non expanding})
	Let $\gamma = \{ \gamma_j: \{1,2, \dots, k\} \to \mathbb C\}_{j=0,1,\dots, J}$ be a kernel satisfying \eqref{multir}. Then, for $f\in L^2(G)$, 
	\begin{equation*}
			\|S[P_J]f\|_{2,J} \leq \|f\|\,.
	\end{equation*}
\end{lemma}
\begin{proof}
   From Lemma \ref{lem:3.2.2}, for all $M\in \mathbb N$
   \begin{align*}
   	\sum_{m=0}^M \sum_{p\in \Lambda_J^m}\|S[p]f\|^2 & = \sum_{m=0}^M \sum_{p\in \Lambda_J^m}\|U[p]f\|^2 - \sum_{m=0}^M \sum_{p\in \Lambda_J^{m+1}}\|U[p]f\|^2 \\
   	& = \|f\|^2 -  \sum_{p\in \Lambda_J^{M+1}}\|U[p]f\|^2 \leq \|f\|^2\,.
   \end{align*}
Taking the limit when $M$ goes to infinity we obtain the result.
\end{proof}
 
 \subsection{Stability}

We now prove that, under condition \eqref{multir}, if two sets of data $f, f' \in L^2(G)$ are closed in the $L^2(G)$ norm, then the $G$-scattering transforms of $f$ and $f'$ are also close in the $\ell_2 (P_J,L^2(G))$ norm. Observe that if the $G$-scattering transform were a linear operator (which is not), this will easily follow form Lemma \ref{lemma4-1-New}.
\begin{proposition} ({\bf Stability}) \label{prop Stability}
Let $\gamma = \{ \gamma_j: \{1,2, \dots, k\} \to \mathbb C\}_{j=0,1,\dots, J}$ be a kernel satisfying \eqref{multir}. Then, the scattering transform is stable in the following sense: if $f, f' \in L^2(G)$, 
\begin{equation}
    \|S[P_J]f-S[P_J]f'\|_{2,J}\leq \|f-f'\|.
\end{equation}
\end{proposition}
\begin{proof}
Let $p\in \Lambda_J^m, m \geq 0.$ Apply 
condition \eqref{multir} and  Remark \ref{remark2.7} with $A = B = 1$ to the function $U[p]f - U[p]f' \in L^2(G)$ to obtain
\begin{equation} \label{eq:new36}
    \|U[p]f-U[p]f'\|^2=\sum_{j=1}^{J}\|\psi_j*(U[p]f-U[p]f')\|^2 + \|\phi*(U[p]f-U[p]f')\|^2.
\end{equation}
By the triangle inequality, for $j=1,2, \dots, J,$
$$
|\psi_j*(U[p]f-U[p]f')| \geq |\psi_j*U[p]f|- |\psi_j*U[p]f'|\,.
$$   
Summing the equality \eqref{eq:new36} over $p\in \Lambda_J^m$ we obtain 
\begin{align*}
    &\sum_{p\in \Lambda_J^m}\|U[p]f-U[p]f'\|^2\\
    &=\sum_{p\in \Lambda_J^m}\left(\sum_{j=1}^J\|\psi_j * U[p]f-\psi_j*U[p]f'\|^2 + \|\phi*U[p]f-\phi*U[p]f'\|^2\right)\\
    &\geq \sum_{p\in \Lambda_J^m}\left(\sum_{j=1}^{J}\| |\psi_j*U[p]f|-|\psi_j*U[p]f'|\|^2 + \|\phi*U[p]f-\phi*U[p]f'\|^2\right)\\
    &= \sum_{p\in \Lambda_J^m}\sum_{j=1}^J \|U[p+j]f-U[p+j]f'\|^2+ \sum_{p\in \Lambda_J^m}\|S[p]f-S[p]f'\|^2\\
    &=\sum_{p\in \Lambda_J^{m+1}}\|U[p]f-U[p]f'\|^2+ \sum_{p\in \Lambda_J^m}\|S[p]f-S[p]f'\|^2.
\end{align*}
Summing the previous inequality over $0\leq m \leq M, M\in \mathbb N$, we get
$$
\sum_{m=0}^M \sum_{p\in \Lambda_J^m}\|U[p]f-U[p]f'\|^2 - \sum_{m=0}^{M} \sum_{p\in \Lambda_J^{m+1}}\|U[p]f-U[p]f'\|^2 \geq \sum_{m=0}^M \sum_{p\in \Lambda_J^m}\|S[p]f-S[p]f'\|^2.
$$
Canceling equal terms in the left hand side of this inequality we obtain
\begin{align*}
\sum_{m=0}^M \sum_{p\in \Lambda_J^m}\|S[p]f-S[p]f'\|^2 &\leq \|f - f'\|^2 - \sum_{p\in \Lambda_J^{M+1}}\|U[p]f-U[p]f'\|^2 \\
& \leq \|f - f'\|^2\,.
\end{align*}
Taking the limit when $M$ goes to infinity we deduce the result.
\end{proof}
\subsection{Energy preservation} \label{Subsec:4-2}
Equality in Lemma \ref{lemma4-1-New} cannot hold in general for a kernel $\gamma$ satisfying 
\eqref{multir}. For example, if $\gamma_0(r)=0$ for all $r=1,2, \dots, k$, so that \eqref{multir} becomes $\displaystyle \sum_{j=1}^J |\gamma_j(r)|^2 = 1, r=1,2, \dots, k$, then $\phi(x)=0$ for all $x\in G$. Therefore, $\phi*f \equiv 0$ for any $f\in L^2(G)$ and $ S[p]f \equiv 0$ for all $p\in P_J$. This leads to the following definition.
\begin{definition} \label{admissible1}
	A kernel $\gamma = \{ \gamma_j: \{1,2, \dots, k\} \to \mathbb C\}_{j=0,1,\dots, J}$ is called {\bf admissible} if in addition to \eqref{multir} it holds that
	\begin{equation} \label{admissible2}
		\beta_\gamma := \min_{1\leq r \leq k} |\gamma_0 (r)|^2 > 0\,.
	\end{equation}
\end{definition}
Under the admissibility condition \eqref{admissible2} it can be proved that the $G$-scattering transform, $S[P_J]: L^2(G) \to \ell_2(P_J, L^2(G))$ is inyective. In fact, if $f, f'\in L^2(G)$ and $S[P_J]f = S[P_J]f'$, in particular $\phi*f = \phi*f'$. Taking Fourier transforms in both sides and using \eqref{eq:FourierGwavelet} we obtain $\gamma_0 (r) \widehat f (\pi^r) = \gamma_0 (r) \widehat {f'} (\pi^r)$ for $r=1,2,\dots, k$. Since $\beta_\gamma >0$, $\gamma_0(r) \neq 0$ for $r=1,2,\dots, k$. Therefore, $\widehat f (\pi^r) = \widehat {f'} (\pi^r), r=1,2,\dots, k$. By the Fourier inversion formula (see \cite{terras_1999}, Chapter 15, Theorem 2), we conclude $f=f'$ in $L^2(G).$

We can quantify this inyectivity. In fact, by the definition of $\| \ \|_{2,J}$ for $f, f' \in L^2(G)$ we can write
	$$
	\|S[P_J]f - S[P_J]f'\|_{2,J}^2 \geq \| \phi*f - \phi*f'\|^2 = \| \phi*(f - f')\|^2\,.
	$$
We now use Plancherel identity \eqref{eq:Plancherel} twice and condition \eqref{admissible2} to obtain
$$
\|S[P_J]f - S[P_J]f'\|_{2,J}^2 \geq \sum_{r = 1}^k\sum_{i,j=1}^{d_r} d_r |\gamma_0(r)|^2 \left|(\widehat{f}-\widehat{f'})(\pi^r)_{ij}\right|^{2} \geq \beta_\gamma \|f - f'\|^2\,.
$$

More important, under the admissibility condition given in Definition \ref{admissible1}, we are able to prove that the $G$-scattering transform preserves energy, that is, equality holds in Lemma 
\ref{lemma4-1-New}. Before proving this, we show a result about the energy decay rate of the scattering propagator $U$ at each level $m\in \mathbb N$.
\begin{proposition} \label{pro-4-4-New}
	Let $\gamma = \{ \gamma_j: \{1,2, \dots, k\} \to \mathbb C\}_{j=0,1,\dots, J}$ be an admissible kernel as given in Definition \ref{admissible1}. Then, for each $m=1,2,\dots$, the scattering propagator $U$ has norm decay rate $\alpha = 1-\beta_\gamma < 1$ at the $m$-th layer, meaning that for all $f\in L^2(G)$,
	\begin{equation} \label{eq:decay2}
		\sum_{p\in \Lambda_J^m} \|U[p]f\|^2 \leq \alpha \sum_{p\in \Lambda_J^{m-1}} \|U[p]f\|^2\,.
	\end{equation}
\end{proposition}
\begin{proof}
	For any path $p\in P_J$ and $f\in L^2(G)$, by definition of $S[p]f$ and Plancherel equality \eqref{eq:Plancherel} we can write
	\begin{equation*}
		\|S[p]f\|^2=\|\phi*U[p]f\|^2= \sum_{r=1}^k\sum_{i,j=1}^{d_r} d_r |\gamma_0(r)|^2  \left|\widehat{U[p]f}(\pi^r)_{i,j}\right|^2.
	\end{equation*}
Since $|\gamma_0(r)|^2 \geq \beta_\gamma$ for all $r=1,2,\dots,k$ we get
\begin{align*}
	\|S[p]f\|^2 &\geq \beta_\gamma {\sum_{r=1}^k\sum_{i,j=1}^{d_r} d_r\left| (\widehat{U[p]f}(\pi^r))_{i,j}\right|^2}\\
	&=\beta_\gamma \|U[p]f\|^2\,,
\end{align*}
where we have used again Plancherel equality \eqref{eq:Plancherel}. Summing the above inequality over all $p\in \Lambda_J^{m-1}$ we have
\begin{equation*}
	\sum_{p\in \Lambda_J^{m-1}} \|S[p]f\|^2 \geq \beta_\gamma \sum_{p\in \Lambda_J^{m-1}}\|U[p]f\|^2.
\end{equation*}
Using the structural property of the scattering network given in Lemma \ref{lem:3.2.2} we obtain
\begin{align*}
	\sum_{p\in \Lambda_J^{m}}\|U[p]f\|^2 &= \sum_{p\in \Lambda_J^{m-1}}\|U[p]f\|^2-\sum_{p\in \Lambda_J^{m-1}}\|S[p]f\|^2 \\
	&\leq (1-\beta_\gamma) \sum_{p\in \Lambda_J^{m-1}}\|U[p]f\|^2 = \alpha \sum_{p\in \Lambda_J^{m-1}}\|U[p]f\|^2\,. \qedhere
\end{align*}
\end{proof}

As a corollary we can prove that the energy carried out by the extracted features of the $G$-scattering transform at level $m, m\geq 1,$ decreases exponentially in the presence of the admissibility condition.
\begin{corollary} \label{corollary2}
Let $\gamma = \{ \gamma_j: \{1,2, \dots, k\} \to \mathbb C\}_{j=0,1,\dots, J}$ be an admissible kernel as given in Definition \ref{admissible1}. Then, for each $m=1,2, \dots$ and each $f\in L^2(G)$,
$$
\sum_{p\in \Lambda_J^m} \|S[p]f\|^2 \leq \left(\max_{1\leq r \leq k} |\gamma_0(r)|^2\right) \alpha^m \|f\|^2\,,
$$
where $\alpha = 1 -\beta_\gamma <1$ and $\beta_\gamma$ is defined in \eqref{admissible2}.	
\end{corollary}
\begin{proof}
Since $S[p]f= \phi*f$ for each $p\in \Lambda_p^m$, by Plancherel equality \eqref{eq:Plancherel} and \eqref{eq:FourierGwavelet}
\begin{align*}
    \|S[p]f\|^2 &= \sum_{r=1}^k \sum_{i,j=1}^{d_r} d_r \left|\widehat {\phi*Uf}(\pi^r)_{i,j}\right|^2\\
    &= \sum_{r=1}^k \sum_{i,j=1}^{d_r} d_r |\gamma_0(r)|^2 \left|\widehat {Uf}(\pi^r)_{i,j}\right|^2\\
    &\leq \left(\max_{1\leq r \leq k} |\gamma_0(r)|^2\right) \|U[p]f\|^2\,.
\end{align*}
Summing over all $p\in \Lambda_J^m$ and using \eqref{eq:decay2} $m$ times, the result follows.
\end{proof}

\begin{theorem} ({\bf Energy preservation}) \label{th:energy}
	Let $\gamma = \{ \gamma_j: \{1,2, \dots, k\} \to \mathbb C\}_{j=0,1,\dots, J}$ be an admissible kernel as given in Definition \ref{admissible1}. Then, for each $f\in L^2(G)$,
	$$
	\|S[P_J]f\|_{2,J} = \|f\|\,.
	$$
\end{theorem}
\begin{proof}
	By the structural property of the scattering network given in Lemma \ref{lem:3.2.2}, for all $M\in \mathbb N$,
	\begin{align} \label{eq:energy2}
		\sum_{m=0}^M \sum_{p\in \Lambda_J^m} \|S[p]f\|^2 &= \sum_{m=0}^M \sum_{p\in \Lambda_J^m} \|U[p]f\|^2 - \sum_{m=0}^M \sum_{p\in \Lambda_J^{m+1}} \|U[p]f\|^2 \nonumber \\ 
		&= \|f\|^2 -  \sum_{p\in \Lambda_J^{M+1}} \|U[p]f\|^2\,.
	\end{align}
    By Proposition \ref{pro-4-4-New}, since $\alpha = 1-\beta_\gamma < 1$,
    $$
    0\leq \lim_{M\to \infty} \sum_{p\in \Lambda_J^{M+1}} \|U[p]f\|^2 \leq \lim_{M\to \infty} \alpha^M \|f\|^2 = 0\,.
    $$
    Taking limits when $M\to \infty$ in \eqref{eq:energy2} we obtain the result.
\end{proof}

\subsection{Equivariance and approximate invariance} \label{Subsec:4-3}
Consider the left and right regular representations of the finite group $G$ on $L^2(G)$. These are
$$
L: G \to \mathcal U(L^2(G)), \quad L_gf(x) = f(g^{-1}x)
$$
and
$$
R: G \to \mathcal U(L^2(G)), \quad R_gf(x) = f(xg),
$$	
for $g,x \in G$ and $f\in L^2(G).$

For $\psi\in L^1(G)$ consider the convolution operator $C_\psi : L^2(G) \to L^2(G)$ given by $C_\psi (f) = \psi * f.$ This convolution operator is {\bf equivariant} (also termed {\bf covariant} in some literature) with respect to the right regular representation; that is, for $f\in L^2(G)$,
\begin{equation} \label{eq_right-cov}
	C_\psi (R_g f)= R_g C_\psi(f), \quad  g\in G\,.
\end{equation}
This follows inmediately from the definition of convolution given in \eqref{def:convolution}.

If $\psi\in L^1(G)$ is also a class function (see \eqref{eugenio:2}), $C_\psi$ is equivariant with respect to the left regular representation, that is, for $f\in L^2(G)$,
	\begin{equation} \label{eq_left-cov}
		C_\psi (L_g f)= L_g C_\psi(f), \quad  g\in G\,.
	\end{equation}
Indeed, since $\psi$ is a class function, for $x\in G$ use \eqref{eugenio:5} and \eqref{def:convolution} to write
$$
C_\psi(L_g f)(x) = \psi * L_g f(x) = L_g f * \psi  = \frac{1}{|G|} \sum_{y\in G} f(g^{-1}y)\psi(y^{-1}x)\,.
$$
With $y^{-1}x = z$,
$$
C_\psi (L_g f)(x) = \frac{1}{|G|} \sum_{y\in G} f(g^{-1}x z^{-1})\psi(z) = f*\psi (g^{-1}x) = \psi * f (g^{-1}x) = L_g C_\psi(f)\,,
$$
where we have used \eqref{def:convolution} and that $\psi$ is a class function.

It is easy to see that the modulus operator commutes with $R_g$ and $L_g$ for all $g\in G$. Since the scattering networks described in Section \ref{section a} are made with convolution with class functions and modulus, we have the following result.
\begin{proposition} \label{eq:equivariance2}
	Let $\gamma = \{ \gamma_j: \{1,2, \dots, k\} \to \mathbb C\}_{j=0,1,\dots, J}$ be a kernel and $S$ the $G$-scattering operator. For all $g\in G$, $p\in P_J$, and $f\in L^2(G)$,
	$$  a) \  S[p]L_g f = L_g S[p]f\,,\qquad \text{and}  \qquad b) \ S[p]R_g f = R_g S[p]f\,.$$
\end{proposition}

Another property of operators is invariance. For the convolution operators $C_\psi$ and the right and left regular representations this means
\begin{equation} \label{eq:invariance2}
	  a) \ S[p]L_g f = S[p]f\,,\qquad \text{and}  \qquad b) \ S[p]R_g f = S[p]f\,.
\end{equation}
$f\in L^2(G)$ and $g\in G$. If \eqref{eq_right-cov} and part $b)$ of \eqref{eq:invariance2} hold simultaneously, for all $g\in G$ and all $x\in G$,
$$
L_g C_\psi (f)(x) = C_\psi(f)(xg).
$$
Taking $x=e$, the identity of the group $G$, it follows that $C_\psi f(g) = C_\psi (e)$, and thus, $C_\psi(f)$ is a constant function in $G$. Similarly if
\eqref{eq_left-cov} and part $a)$ of \eqref{eq:invariance2} hold simultaneously. Consequently, we cannot have, in general, invariance of the $G$-scattering transform. Nevertheless, we can prove that at the depth of the scattering propagator increases, the $G$-scattering features become closer to be invariant, and, hence, closer to be constant functions.
\begin{theorem} ({\bf Approximate invariance}) \label{th:approxinv}
	Let $\gamma = \{ \gamma_j: \{1,2, \dots, k\} \to \mathbb C\}_{j=0,1,\dots, J}$ be an admissible kernel as given in Definition \ref{admissible1}. For each $m\geq 1, f\in L^2(G), g\in G$,
	$$
	\sum_{p\in \Lambda_J^m} \| S[p]L_g f - S[p]f\|^2 \leq  \left(\max_{1\leq r \leq k} |\gamma_0(r)|^2\right) 4\, \alpha^m \|f\|^2\,,
	$$
	where $\alpha = 1 -\beta_\gamma <1$ and $\beta_\gamma$ is defined in \eqref{admissible2}. Moreover, a similar result holds replacing $L_g$ by $R_g, g\in G.$	
\end{theorem}
\begin{proof}
	By part $a)$ of Proposition \ref{eq:equivariance2},
	\begin{align*}
		\sum_{p\in \Lambda_J^m} \| S[p]L_g f - S[p]f\|^2 &= \sum_{p\in \Lambda_J^m} \|L_g S[p] f - S[p]f\|^2 \nonumber \\
		&\leq \| L_g - I\|^2 \sum_{p\in \Lambda_J^m} \| S[p]f\|^2 
	\end{align*}
    In view of Corollary \ref{corollary2}, we only need to prove the estimate $\|L_g - I\|^2 \leq 4.$ This follows from the following computation. For any $f\in L^2(G)$,
    \begin{align*}
    	\|L_g f - f\|^2 &=\frac{1}{|G|} \sum_{y\in G} |f(g^{-1}y) - f(y)|^2 
    	\leq \frac{1}{|G|} \sum_{y\in G} (|f(g^{-1}y)| + |f(y)|)^2 \\
    	&\leq \frac{2}{|G|} \sum_{y\in G} (|f(g^{-1}y)|^2 + |f(y)|^2) = 4 \|f\|^2\,.
    \end{align*}
The proof for the right regular representation is similar.	
\end{proof}

\

 \section{Relaxing the admissibility condition}\label{relaxation}

If the kernel $\gamma$ does not satisfy \eqref{admissible2}, it holds that $\gamma_0(r)=0$ for some $r\in\{1,2, \dots, k\}.$ Hence, the proof of Proposition \ref{pro-4-4-New} does not give a norm decay rate smaller than 1 for the scattering propagator. In this case, we are going to show how to obtain a norm decay rate $\alpha <1$ for the scattering propagator. To do this, we need to briefly explain the result contained in Theorem 5 of \cite{kueh_olson_rockmore_tan_2001}.

Let $S \subset \{1,2, \dots, k\}$ be a subset of indexes or non-equivalent irreducible unitary representations of a finite group $G$. Consider the representation of $G$ given by 
$$
\rho^S :=\bigoplus_{r\in S} d_r \pi^r\,,
$$
whose character is $\chi^{\rho^S}=\sum_{r\in S} d_r \chi^r$. For a representation $\rho$ of $G$, denote by $\rho^{\dagger}$ the contragredient representation given by $\rho^\dagger (x) = \overline{\rho(x)}^t$ for all $x\in G$.

Consider now the representation $\rho^S\otimes(\rho^S)^{\dagger}$. It has a decomposition in terms of irreducible unitary representations of $G$, of the form
	\begin{equation*}
		\rho^S\otimes(\rho^S)^{\dagger}= \bigoplus_{r=1}^k n_S(r) \pi^r.
	\end{equation*}
Let $T(S) = \{r\in \{1,2, \dots, k\} : n_S(r) \neq 0 \}$. Define $\displaystyle deg(S)= \sum_{r\in S} d_r^2.$

\begin{theorem}[Theorem 5 in \cite{kueh_olson_rockmore_tan_2001}]  \label{th:Kueh}
Let $f\in L^2(G)$ be nonnegative and $S$ a subset of indexes of the inequivalent irreducible representations of $G$. Then
\begin{equation*}
	\sum_{r\in T(S)} d_r \sum_{i,j=1}^{d_r} |(\widehat f (\pi^r))_{i,j}|^2 \geq \frac{\text{deg(S)}}{|G|} \| f\|^2\,.
\end{equation*}
\end{theorem}

\begin{theorem} \label{th:new}
Let $\gamma=\{\gamma_j:\{1,\ldots, k\}\to \CC\}_{0\leq j \leq J}$ be a kernel satisfying \eqref{multir}. Let $S\subset \{1, \ldots, k\}$ and suppose that $\gamma$ is such that 
$$
\beta_\gamma(S) := \frac{\text{deg(S)}}{|G|} \min_{r\in T(S)} |\gamma_0(r)|^2 > 0\,.
$$

Then, for $f\in L^2(G)$ and $m=2, 3, \dots,$ 
$$
\sum_{p\in \lambda_J^m} \|U[p]f\|^2 \leq \alpha (S) \sum_{p\in \lambda_J^{m-1}} \|U[p]f\|^2\,,
$$
where $\alpha(S):= 1 - \beta_\gamma (S) < 1.$
\end{theorem}
\begin{proof}

For any path $p\in P_J,\ p\neq \emptyset $, and $f\in L^2(G)$, by definition of $S[p]f$ and Plancherel equality \eqref{eq:Plancherel} we can write
\begin{align*}
	\|S[p]f\|^2 &=\|\phi*U[p]f\|^2= \sum_{r=1}^k\sum_{i,j=1}^{d_r} d_r |\gamma_0(r)|^2  \left|\widehat{U[p]f}(\pi^r)_{i,j}\right|^2 \\
	&\geq \sum_{r\in T(S)}\sum_{i,j=1}^{d_r} d_r |\gamma_0(r)|^2  \left|\widehat{U[p]f}(\pi^r)_{i,j}\right|^2 \\
	&\geq (\min_{r\in T(S)} |\gamma_0(r)|^2) \sum_{r\in T(S)}\sum_{i,j=1}^{d_r} d_r  \left|\widehat{U[p]f}(\pi^r)_{i,j}\right|^2 \\
	&\geq \beta_\gamma(S) \|U[p]f\|^2\,,
\end{align*}
where the last inequality is due to Theorem \ref{th:Kueh} since $U[p]f$ is nonnegative. The proof now proceeds as in the proof of Proposition
\ref{pro-4-4-New}.
\end{proof}

\begin{remark}
	Under the same conditions of Theorem \ref{th:new}, the same inequality as in Corollary \ref{corollary2} can be proved replacing $\alpha$ by $\alpha(S).$ Moreover, under these conditions we also have the energy preservation result stated in Theorem \ref{th:energy}.
\end{remark}

Note that $S$ can be different or equal to $T(S)$, as shown in the following examples.  

\begin{example}
Let $G=\ZZ/N\ZZ$. A system of irreducible characters of $G$ is $\{e^{-\frac{2\pi i}{N}\lfloor\frac{ N }{2}\rfloor (\cdot)}, \ldots, e^{\frac{2\pi i}{N}\lfloor\frac{ N - 1}{2}\rfloor (\cdot)}\}$. Let $n \leq \frac{N-1}{2}$, and let $S = \{-\lfloor\frac{ n }{2}\rfloor , \dots, \lfloor\frac{ n-1 }{2}\rfloor\}$ be the indexes of the characters $\{e^{-\frac{2\pi i}{N}\lfloor\frac{ n }{2}\rfloor(\cdot)}, \ldots, e^{\frac{2\pi i}{N}\lfloor\frac{ n-1 }{2}\rfloor(\cdot)}\}$. Then  \begin{equation}
    \chi^{\rho^{m_S}}(x)=\sum_{l=-\lfloor\frac{ n }{2}\rfloor}^{\lfloor\frac{ n -1}{2}\rfloor} e^{\frac{2\pi i l x}{N}},
\end{equation}
so
\begin{equation}
    \chi_{\rho^{m_S}}(x) \overline{\chi_{\rho^{m_S}}(x)} = \sum_{l=-\lfloor\frac{ n }{2}\rfloor}^{\lfloor\frac{ n-1 }{2}\rfloor} \sum_{m=-\lfloor\frac{ n }{2}\rfloor}^{\lfloor\frac{ n -1}{2}\rfloor} e^{\frac{2\pi i (l-m) x}{N}} 
\end{equation}
Since $l-m$ varies from $-n+1$ to $n-1$, thus $T(S) \neq S$, because $T(S) = \{-n+1,\dots,n-1\}$ contains all the indexes of $\{e^{-\frac{2\pi i}{N}(n-1)(\cdot)},\ldots, e^{\frac{2\pi i}{N} (n-1)(\cdot)}\}$ and has $2n-1$ elements, while $S$ has $n$ elements.
\end{example}

\begin{example}
As another example, if $G=S_n$ and $S$ contains the indexes of the trivial character and the character given by the sign of the permutation ($\chi^1 = I$ and $\chi ^2 =  \operatorname{sign}$) then $deg(S)=2$ and 
\begin{equation}
    \rho_{m_S}(x)= \begin{bmatrix}
        1 & 0\\
        0 & \operatorname{sign}(x)
    \end{bmatrix},
\end{equation}
so that 
\begin{equation}
    \rho_{m_S}(x)\otimes (\rho_{m_S})^{\dagger}(x)= \begin{bmatrix}
        1 & 0 & 0 & 0\\
        0 & \operatorname{sign}(x) & 0 & 0\\ 0 &  0 &\operatorname{sign}(x) &0\\ 0 &0&0&1
    \end{bmatrix},
\end{equation}
thus $T(S)=S$.
\end{example}


\section{Numerical results} \label{numerical results}
In this section, we present an implementation of the $G$-scattering transform, focused on classification problems.
\subsection{The MNIST dataset}\label{mnist}

The MNIST dataset \cite{PyTorch} consists of $28 \times 28$ gray scale images of handwritten digits ranging from 0 to 9 (see Figure \ref{fig:mnist}). It has been used by several authors (see, for example, \cite{lecun1998gradient} and \cite{zou}) to test different classification methods.
\begin{figure}[H]
    \centering
    \includegraphics[width=0.5\linewidth]{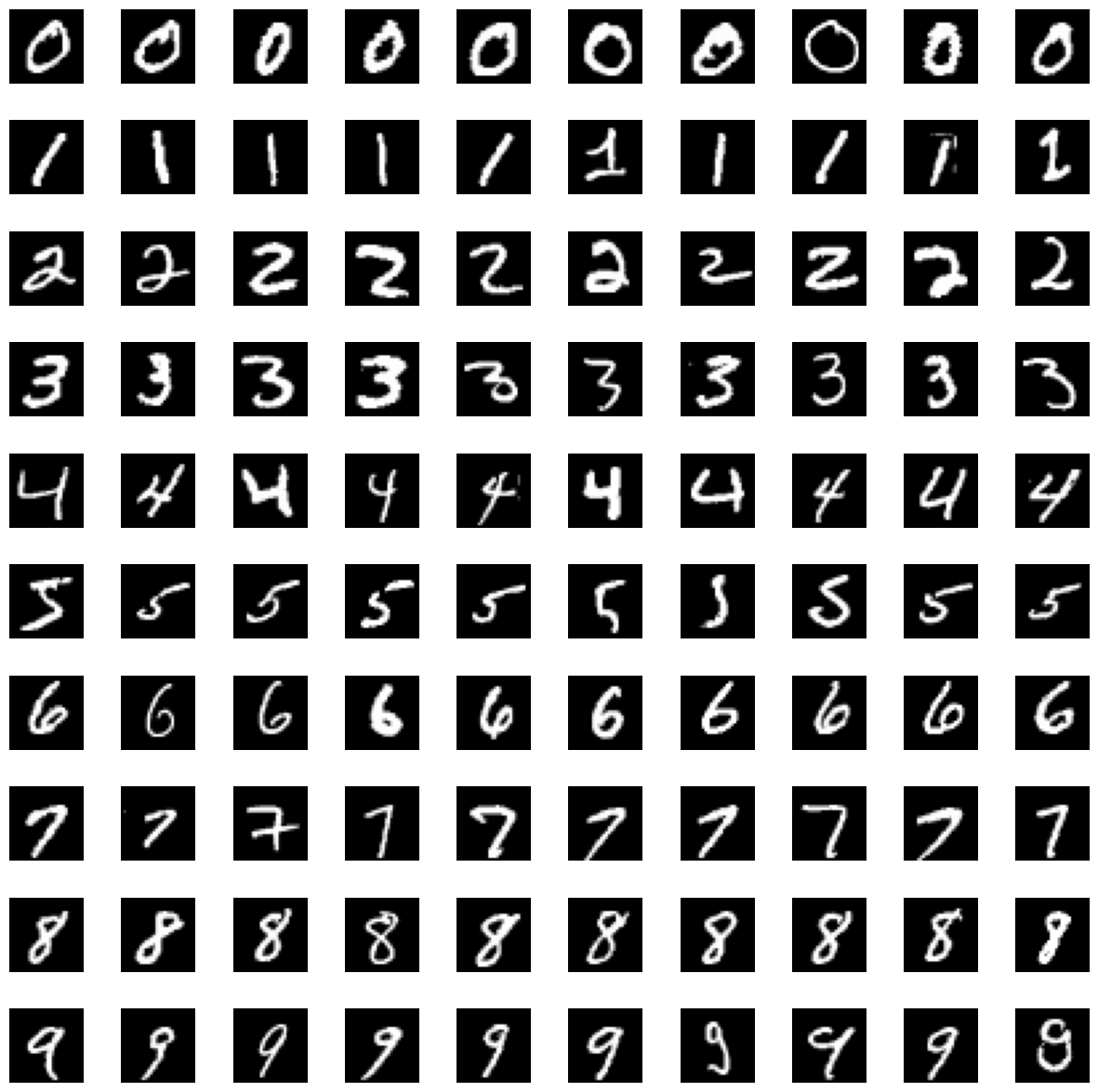}
    \caption{Example of data}
    \label{fig:mnist}
\end{figure}

There are 60,000 training
images and 10,000 testing images in total, where the task is to classify the images according to the digits. In this work, we use a subset of 4,800 training images and 2,400 test images. We model the images as functions over the group $ G = \mathbb{Z}/28\mathbb{Z} \times \mathbb{Z}/28\mathbb{Z}$ and apply a $G$-Scattering Transform with varying type and numbers of filters and depths. We compare our results with those of Zou and Lerman \cite{zou}, who constructed a graph representing
the underlying grid of the images and applied their proposed graph scattering.

\subsubsection{$G$-Scattering Transform}

Each image is interpreted as a function $f: \mathbb{Z}/28\mathbb{Z} \times \mathbb{Z}/28\mathbb{Z} \to \mathbb{R}$, where pixel intensities, normalized between $[-1, 1]$, define the function values. For $(a,b)\in \mathbb{Z}/28\mathbb{Z} \times \mathbb{Z}/28\mathbb{Z}$ the character $\chi_{(a,b)}$
is given by
$$ 
\chi_{(a,b)} (n,m) = \omega^{na+mb}\,, \qquad (n,m) \in \mathbb{Z}/28\mathbb{Z} \times \mathbb{Z}/28\mathbb{Z}\,,
$$
where \( \omega = e^{2\pi i/28} \) is a primitive 28th root of unity. Gray scale images of the real and imaginary parts of theses $28\times 28 =784$ characters are given in Figure \ref{fig:characters}.


\begin{figure}[H]
    \centering
    \begin{subfigure}{0.45\textwidth}
        \centering
        \includegraphics[width=\linewidth]{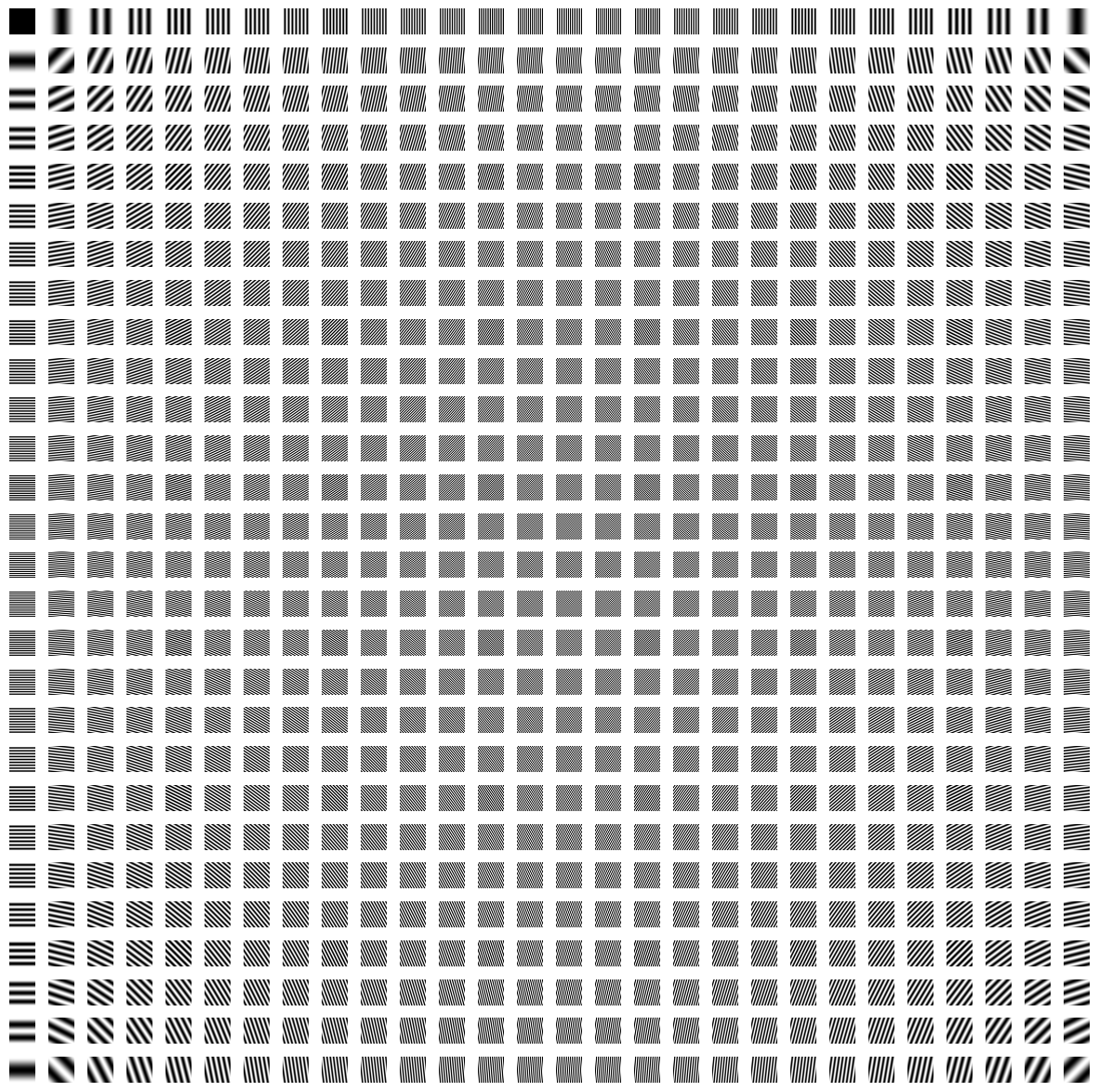}
        \caption{Real part of the characters}
    \end{subfigure}
    \hfill
    \begin{subfigure}{0.45\textwidth}
        \centering
        \includegraphics[width=\linewidth]{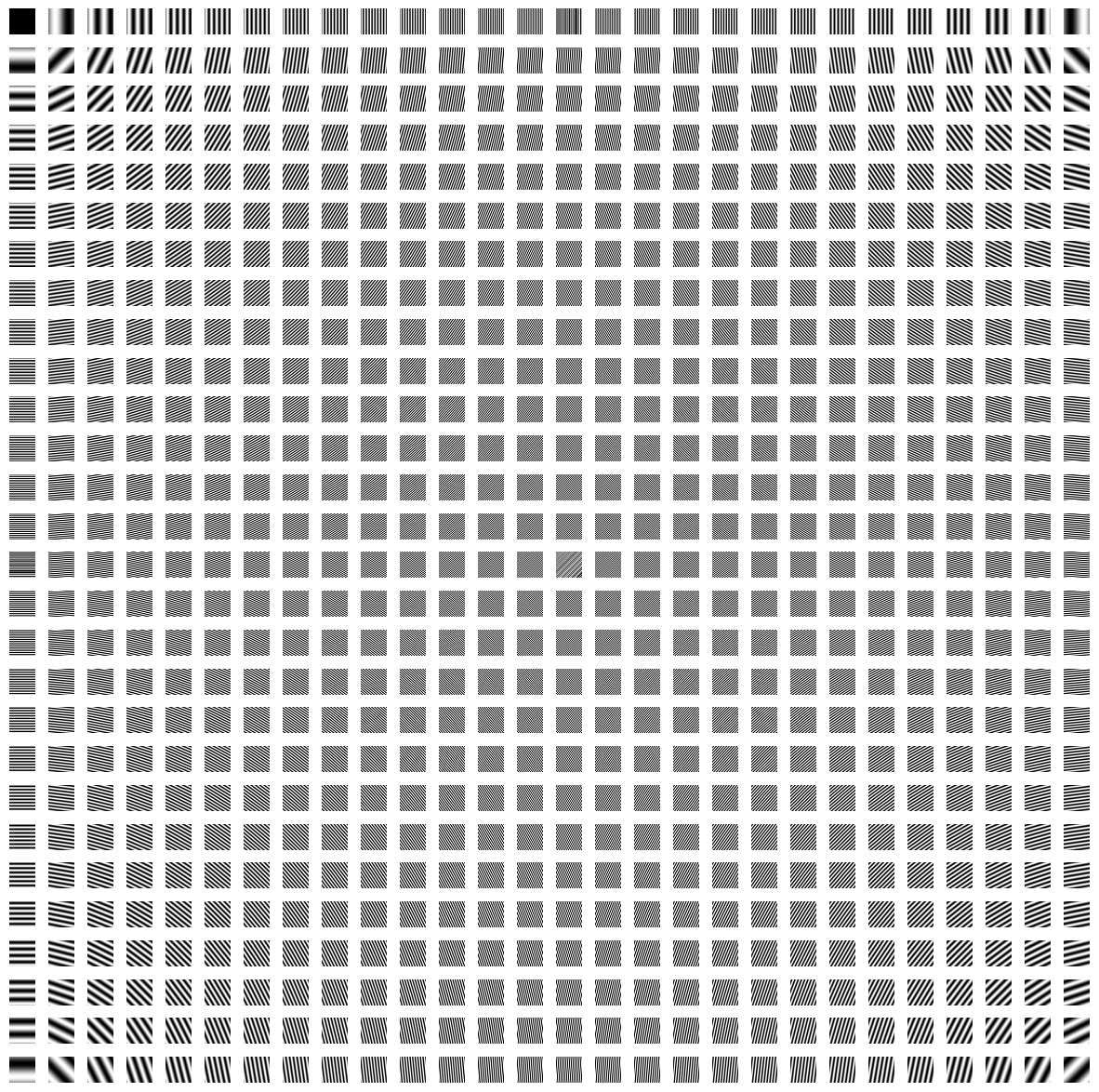}
        \caption{Imaginary part of the characters}
    \end{subfigure}
    \caption{Characters of $\mathbb{Z}/28\mathbb{Z}\times \mathbb{Z}/28\mathbb{Z}$} 
    \label{fig:characters} 
\end{figure}

We apply a $G$-Scattering Transform with 1, 2, and 3 layers and filter scales  $J=1$, $J=5$, and $J=8$. 
We fix a mexican hat wavelet with $\sigma = 2.0$ (see Figure \ref{fig:mexhat}) as the low pass filter 
\[
\widetilde{\phi}(x, y; \sigma)=\widetilde{\psi_0}(x, y; \sigma) = \left(1 - \frac{x^2 + y^2}{2\sigma^2} \right) \exp\left(-\frac{x^2 + y^2}{2\sigma^2} \right)\,.
\]

\begin{figure}[H]
	\centering
	\includegraphics[width=0.4\linewidth]{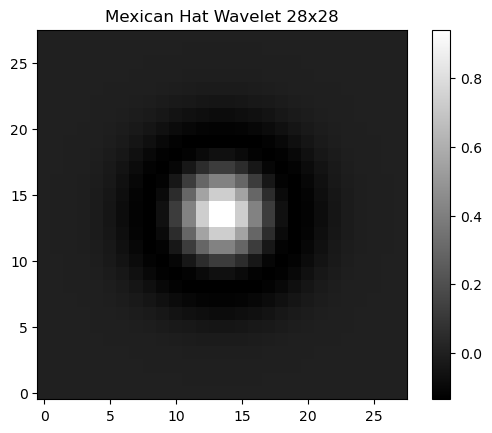}
	\caption{Mexican hat wavelet}
	\label{fig:mexhat}
\end{figure}

The families of wavelets are based on a mother wavelet rotated and dilated. The mother wavelets used are:
\begin{itemize}
    \item Shannon wavelet (see Figure \ref{fig:shannon})
\[
\widetilde{\psi}(x, y; f_x, f_y) = 
\sin(f_x x)sinc(x)\sin(f_y y)sinc(y),
\]
    \item Daubechies wavelets $db2$ (see Figure \ref{fig:db2}) (based on the python module \verb|pywt|)

    
\end{itemize}

\begin{figure}[H]
    \centering
    \includegraphics[width=1\linewidth]{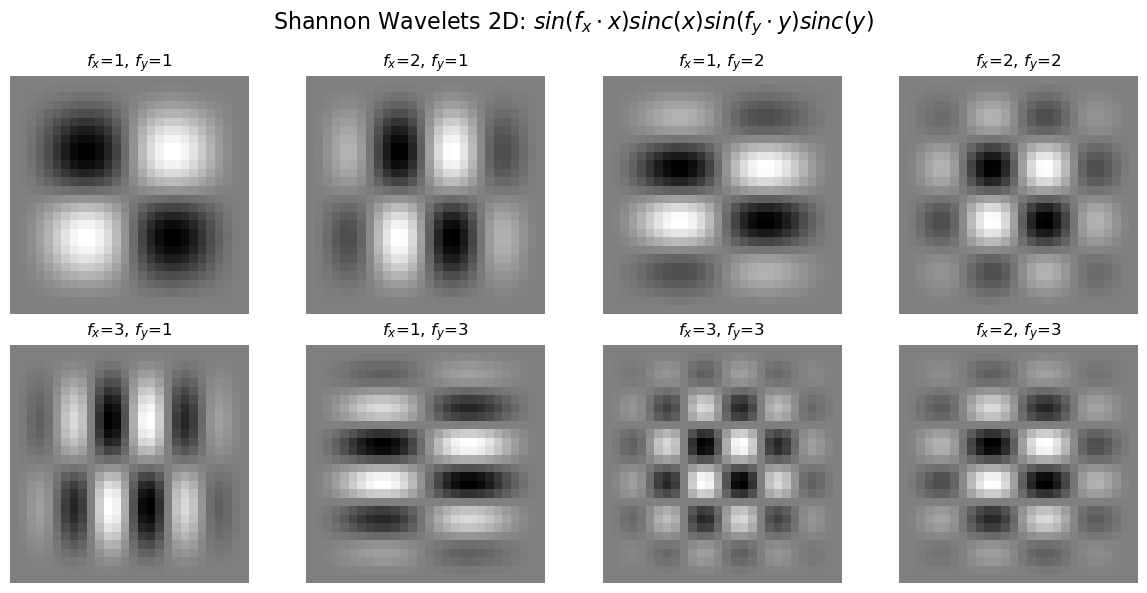}
    \caption{Shannon Wavelets}
    \label{fig:shannon}
\end{figure}



\begin{figure}[H]
	\centering
	\includegraphics[width=1\linewidth]{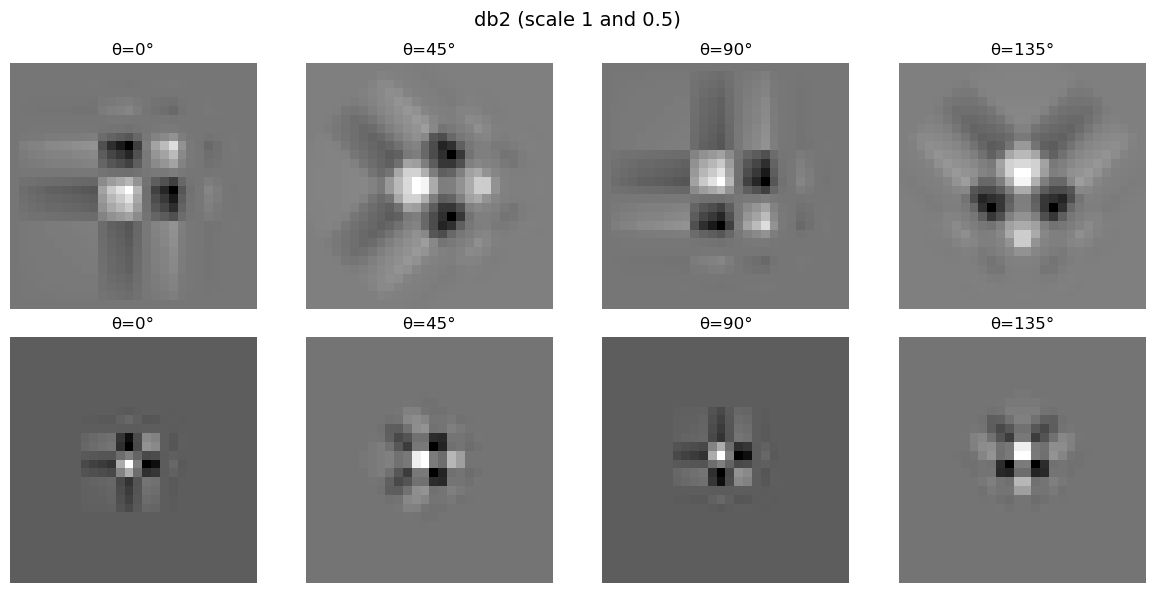}
	\caption{Daubechies wavelet $db2$ rotated by $\theta$ (top row) and dilated by $0.5$ and rotated by $\theta$ (bottom row).}
	\label{fig:db2}
\end{figure}


We normalize the functions $\{ \widetilde{\phi}, \widetilde{\psi_j}: j \in 1, \ldots, J\}$ in order to satisfy the unitary conditions \eqref{multir}. To do this, write
$$
\widetilde {\phi}(n,m)= \sum_{(a,b)\in\mathbb{Z}/28\times \mathbb{Z}/28\mathbb{Z}} \widetilde{\gamma_0}(a,b) \chi_{(a,b)}(n,m)
$$
and
$$
\widetilde {\psi_j}(n,m)= \sum_{(a,b)\in\mathbb{Z}/28\times \mathbb{Z}/28\mathbb{Z}} \widetilde{\gamma_j}(a,b) \chi_{(a,b)}(n,m)\,, \quad j=1, \dots, J.
$$
For $j=0, 1, \dots, J$ and $(a,b)\in \mathbb{Z}/28\times \mathbb{Z}/28\mathbb{Z}$ define
$$
\gamma_j(a,b) := \frac{\widetilde{\gamma_j}(a,b)}{\left(\sum_{j=0}^J |\widetilde{\gamma_j}(a,b)|^2\right)^{1/2}}\,,
$$
so that the condition \eqref{multir} is satisfy for the kernel $\{ \gamma_j (a,b): (a,b)\in \mathbb{Z}/28\times \mathbb{Z}/28\mathbb{Z} \}_{0\leq j \leq J}$. In our experiments, we use the low pass filter
$$
\phi(n,m) := \sum_{(a,b)\in\mathbb{Z}/28\times \mathbb{Z}/28\mathbb{Z}} \gamma_0(a,b) \chi_{(a,b)}(n,m)
$$
and the wavelets
$$
\psi_j(n,m) := \sum_{(a,b)\in\mathbb{Z}/28\times \mathbb{Z}/28\mathbb{Z}} \gamma_j(a,b) \chi_{(a,b)}(n,m)\,, \quad j=1, \dots, J.
$$

Figures \ref{fig:mexhat}, \ref{fig:shannon}, and \ref{fig:db2} represent Mexican hat, Shannon, and Daubechies 2 wavelets that have already been normalized according to the procedure described just described.

Figures \ref{fig:zerol1} and \ref{fig:zerol2} represent the first level scattering of an image of numbers 0 and 7 from the dataset associated to Shannon and Daubechies $db2$ wavelets with $J=8$.

\begin{figure}[H]
    \centering
    \includegraphics[width=1\linewidth]{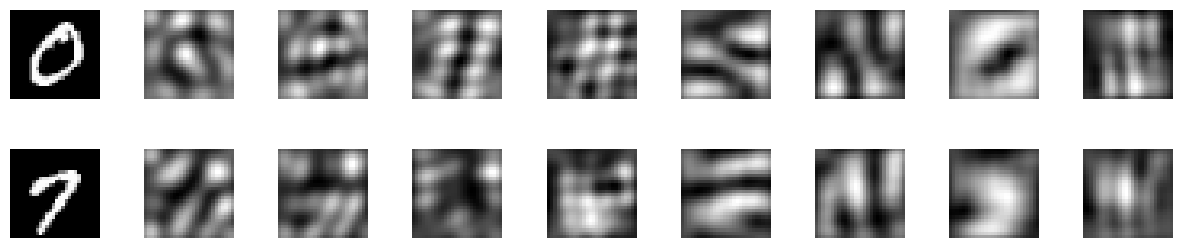}
    \caption{First level scattering associated to Shannon wavelets $\phi*|\psi_j*f|, j\in \{1,\ldots, 8\}$ for the image of number $0$ (top row) and number $7$ (bottom row).}
    \label{fig:zerol1}
\end{figure}

\begin{figure}[H]
    \centering
    \includegraphics[width=1\linewidth]{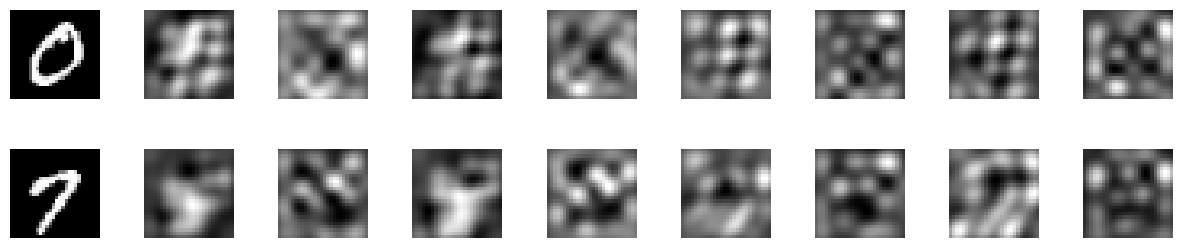}
    \caption{First level scattering associated to Daubechies $db2$ wavelets $\phi*|\psi_j*f|*, j\in \{1,\ldots, 8\}$ for the image of number $0$ (top row) and number $7$ (bottom row).}
    \label{fig:zerol2}
\end{figure}

\subsubsection{Classification and Results}
The resulting feature vectors are subsequently reduced in dimensionality to 1,000 components using Principal Component Analysis (PCA). The extracted features are classified using a one-vs-all Support Vector Machine (SVM) classifier. We evaluate the classification accuracy for different configurations of the $G$-Scattering Transform. Table \ref{tab:results}  presents the classification results over the 2400 test images. Compared with Zou and Lerman SVM over a graph scattering transform \cite{zou} the Shannon wavelets with $M = 2$ layers has an accuracy of $95.68\%$ which is less than our level 2 result for Shannon and $db2$ with $J=8$ ($96.50\%$ and $96.75\%$ respectively). For $M=3$, Zou and Lerman with a SVM classifier \cite{zou} obtained $96.95\%$ accuracy, less than our level 3 results for Shannon with $J=8$ ($97.04\%$) and for $db2$ with $J=1$($96.87\%$), $J=5$($96.91\%$), and $J=8$($97.08\%$).
 
Our results indicate that deeper than $M=3$ scattering transformations do not improve classification performance, with in general three layers and $J=8$ achieving the highest accuracy. 
 
Other wavelets, such as Daubechies wavelet $db4$ and Coifman Wavelet $coif1$,  both based on the python module $pywt$, have been used in our experiments. The results do not improve the results obtained for Daubechies wavelet $db2$ in Table \ref{tab:results}.

\begin{table}[H]
    \centering
    \begin{tabular}{c c c c}
        \toprule
        \textbf{Type} & \textbf{Depth} & \textbf{J} & \textbf{SVM Accuracy (\%)} \\
        \midrule
        Original Images & No Scattering & - & 92.00 \\
        Shannon & 1 Layer  & $J=1$   & 93.75 \\
        Shannon & 1 Layer  & $J=5$   & 94.00 \\
        Shannon & 1 Layer  & $J=8$  & 95.00 \\
        Shannon & 2 Layers & $J=1$   & 95.54 \\
        Shannon & 2 Layers & $J=5$   & 95.58 \\
        Shannon & 2 Layers & $J=8$  &\textbf{ 96.50} \\
        Shannon & 3 Layers & $J=1$   & 95.00 \\
        Shannon & 3 Layers & $J=5$   & \textbf{95.90} \\
        Shannon & 3 Layers & $J=8$  & \textbf{97.04} \\
        \midrule
\midrule
db2 & 1 Layer  & $J=1$   & 93.75 \\
        db2 & 1 Layer  & $J=5$   & 94.50 \\
        db2 & 1 Layer  & $J=8$  & 95.45 \\
        db2 & 2 Layers & $J=1$   & 95.54 \\
        db2 & 2 Layers & $J=5$   & 95.58 \\
        db2 & 2 Layers & $J=8$  & 96.75 \\
        db2 & 3 Layers & $J=1$   & \textbf{96.87} \\
        db2 & 3 Layers & $J=5$   & \textbf{96.91} \\
        db2 & 3 Layers & $J=8$  & \textbf{97.08} \\
%
        
        \bottomrule
    \end{tabular}
    \caption{Classification accuracy (\%) of the $G$-Scattering Transform over de MNIST dataset.}
    \label{tab:results}
\end{table}


%
%

\subsection{Distinguishing between bark or meow}\label{classsound}
The aim of this example is to classify sound files as meow or bark. We use labeled audio samples categorized as either meows (50 files) or barks (46 files) from the \cite{suh_2019_3563990} database. The audios are pulse-code modulated with a bit depth of 16 and a sampling rate of $44.1$ kHz:

\begin{itemize}
	\item Bit-depth $16$: The amplitude of each sample in the audio is $2^{16}$ $(=65536)$ possible values.
	\item Samplig rate = $44.1$ kHz: Each second in the audio consists of $44100$ samples. So, if the duration of the audio file is $3.2$ seconds, the audio will consist of $44100 \cdot 3.2 = 141120$ values.
\end{itemize}

We start by performing the Kaiser Fast transform of the data, using the Python library Librosa \cite{mcfee2015librosa}, with a sampling rate of $16000$ (see Figure \ref{fig:kaiser}). Kaiser Fast transform is a resampling method used for changing the sampling rate of an audio signal. It involves convolving the input signal with a kernel derived from the Kaiser window, leading to a modified signal with a different sampling rate. This step allows a higher computational speed.

\begin{figure}[h!]
	\centering
	\includegraphics[width=0.65\textwidth]{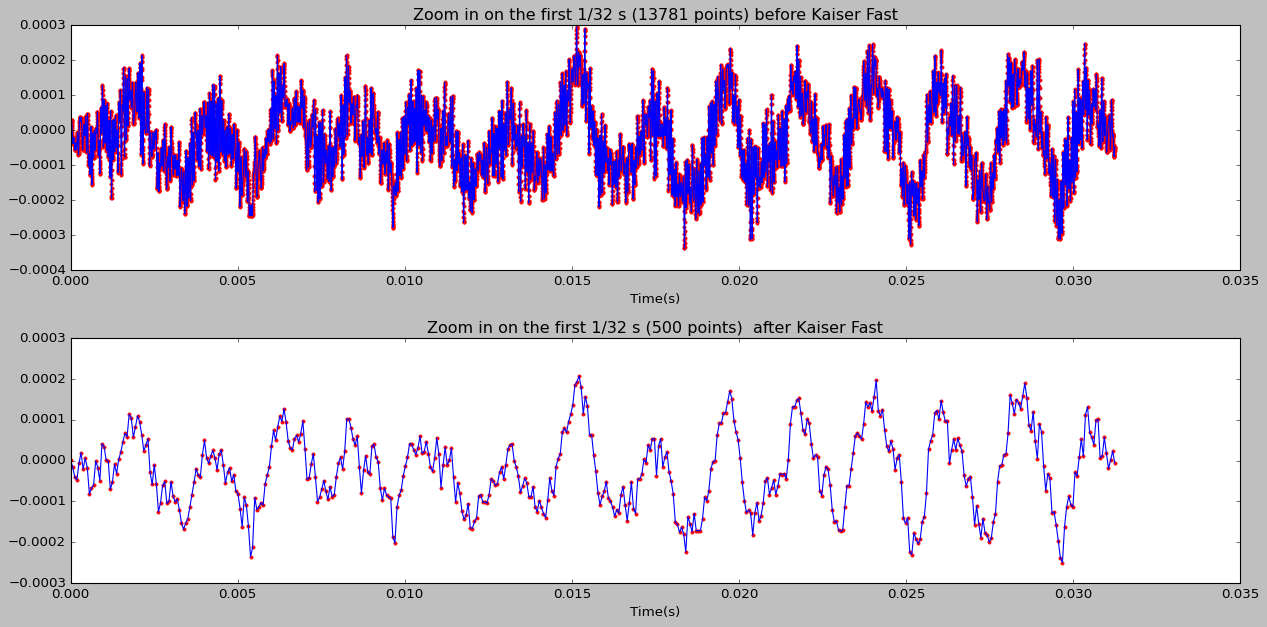} 
	\caption{The original length of an example signal is $6,209,280$ (top row); after the Kaiser Fast transform and the selection of $2$ seconds, we end up with a signal of length $32,000$ (bottom row).}
	\label{fig:kaiser} 
\end{figure}

\

To classify the signals, we choose to take the points from the first 2 seconds of each audio. The amplitude of the signals is then normalized between $[-1/2,1/2]$.
\begin{equation} \label{normalization}
\text{normalized data} = \frac{\text{data} - \min(\text{data})}{\max(\text{data}) - \min(\text{data}) + \epsilon}-\frac{1}{2},
\end{equation}
where $\epsilon$ is a small constant (sucha as $10^{-6}$) added to the denominator to avoid division by zero. Normalizing data makes it easier to interpret the discrete wavelet transform coefficients, so that they can be compared more directly to gauge their relative importance.

\

Figure \ref{fig:n1} shows the first two seconds of 5 meow files and 5 bark files, after performing the Kaiser transform and the normalization \eqref{normalization}. Thus, each signal has 32,000 values. For future reference a Short Time Fourier transform with Kaiser window is performed in the ten images of Figure \ref{fig:n1}, resulting in the spectrograms of Figure \ref{fig:spectrograms}. These sepectrograms have a time duration of 2 seconds and a frequency range of 8,000 hertz.

\begin{figure}[h!] 
\centering
\includegraphics[width=.93\textwidth]{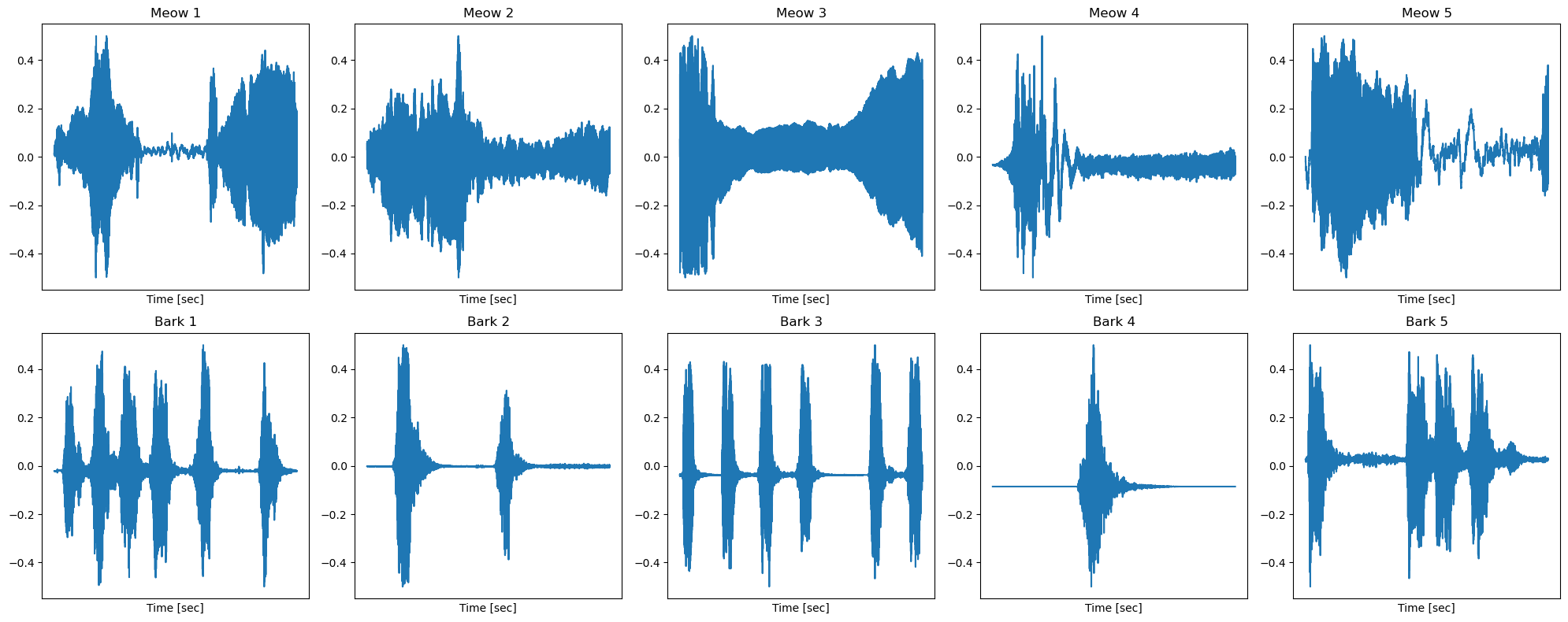}
\caption{First row: data labeled as `Meow`, second row: data labeled and as 'Bark`. Signal's amplitudes are normalized between $[-1/2,1/2]$ and only two seconds are taken.}
\label{fig:n1} 
\end{figure}


\begin{figure}[h!] 
	\centering
	\includegraphics[width=.93\textwidth]{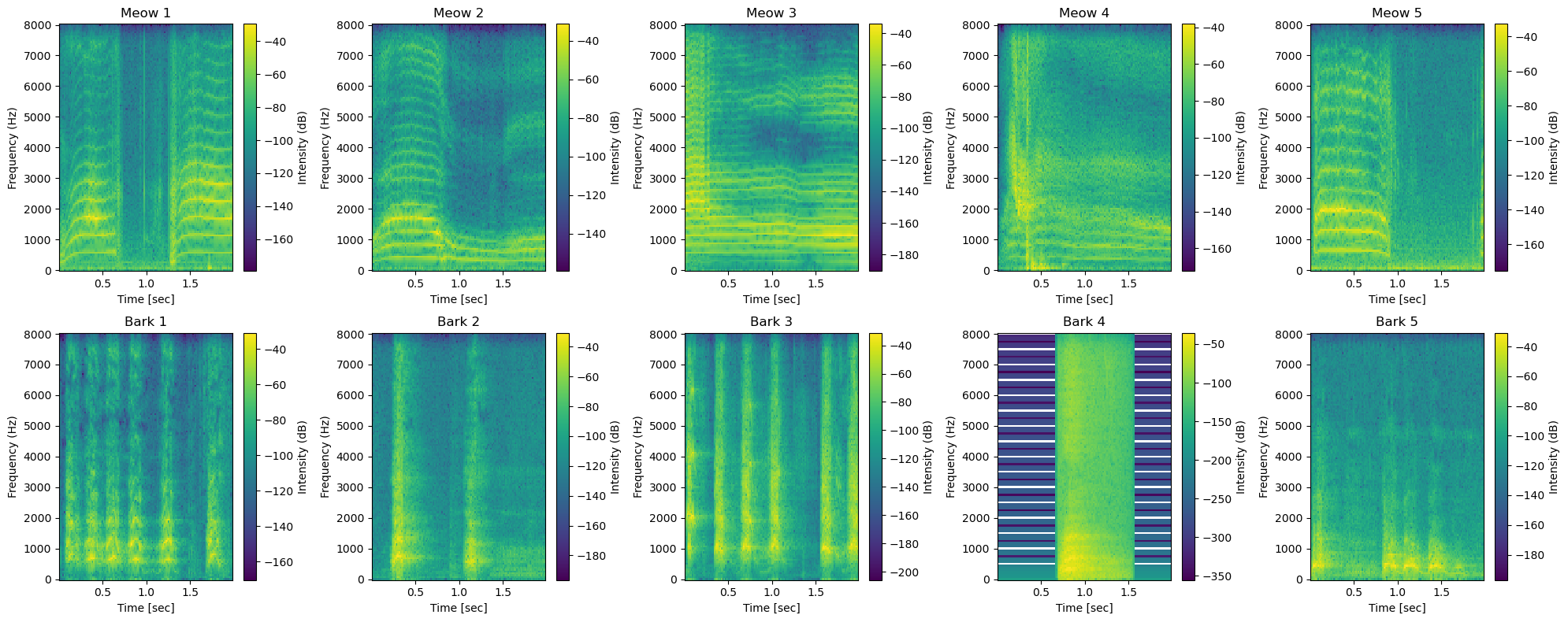} 
	\caption{Spectrograms of the audio exhibit in Figure \ref{fig:n1} created with scipy.signal.}
	\label{fig:spectrograms} 
\end{figure}

\newpage

\subsubsection{The group}\label{ex:oo}
The group we are going to use is the (non-abelian) Affine group over a finite field. Let $p$ be a prime. Consider the group of transformations $x\mapsto ax+b,\, a, b \in\mathbb{F}_p$, $a\neq0$, denoted by $\operatorname{Aff}(\mathbb{F}_p)$, with composition operation. It has $p(p-1)$ elements. 
The elements of $\operatorname{Aff}(\mathbb{F}_p)$ can be identified with matrices of the form
\begin{equation}
    \begin{bmatrix}
        a & b\\
        0 & 1
    \end{bmatrix}, \quad a,b \in \mathbb{F}_p, a\neq 0\,,
\end{equation}
and the operation is multiplication of matrices.
Since
\begin{equation*}
\begin{bmatrix}
        c & d\\
        0 & 1
    \end{bmatrix}
    \begin{bmatrix}
        a & b\\
        0 & 1
    \end{bmatrix}
    \begin{bmatrix}
        c & d\\
        0 & 1
    \end{bmatrix}^{-1} = \begin{bmatrix}
        a & (1-a)d+bc\\
        0 & 1
    \end{bmatrix},
\end{equation*}
the conjugacy classes are
\begin{equation*}
    C^0 = \left \{\begin{bmatrix}
        1 & 0\\
        0 & 1
    \end{bmatrix} \right\}
\end{equation*}
\begin{equation*}
    C^1 = \left \{\begin{bmatrix}
        1 & b\\
        0 & 1
    \end{bmatrix}, \quad b \in \mathbb{F}_p\setminus \{0\}\right\},
\end{equation*}
\begin{equation*}
    C^{a} = \left \{\begin{bmatrix}
        a & b\\
        0 & 1
    \end{bmatrix}, \quad b \in \mathbb{F}_p\right\}, a\in \{2, 3, \ldots, p-1\}.
\end{equation*}
Thus $\operatorname{Aff}(\mathbb{F}_p)$ has $p$ conjugacy classes, and, hence $p$ irreducible unitary representations (see e.g. \cite[Chapt. 15, Theorem 3]{terras_1999}). Let $g$ be a generator of $\mathbb{F}_p^\times$, so that all non-zero elements of $\mathbb{F}_p$ are $g^j$ for some $j\in\{0, \ldots, p-2\}$. By the properties of the matrix product, for $k \in \{0,\ldots, p-2\}$, the homomorphisms below are one-dimensional representations:
\begin{align*}
    \chi^k: &\quad \operatorname{Aff}(\mathbb{F}_p)  \quad \quad\longrightarrow \mathbb{C}\\
    &\begin{bmatrix}
        g^j & b\\
        0 & 1
    \end{bmatrix} \longmapsto \exp{\left(\frac{2\pi i kj}{p-1}\right)}, \quad j\in\{0, \ldots, p-2\}, \ b\in \mathbb{F}_p
\end{align*}
Since $\displaystyle |\operatorname{Aff}(\mathbb{F}_p)|=\sum_{i=0}^{p-1} d_i^2$ (see e.g. \cite[Chapt. 15, Lemma 2]{terras_1999}), the dimension of the last irreducible representation is $p-1$ . Therefore, $\chi^{p-1}(C^0)=p-1$. Using the orthogonality relations for characters (\cite[Chapt. 15, Theorem 3]{terras_1999}), we can complete the character table of $\operatorname{Aff}(\mathbb{F}_p)$:  $\chi^{p-1}(C^1)=-1$ and $\chi^{p-1}(C^{g^k})=0$ for $k\in\{1, \ldots, p-2\}$. See Table  \ref{table:chars S_3}\,.
\begin{table}[h!]
\centering
\begin{tabular}{|l|r|r|r|r|r|}
\hline
	& $C^0$ & $C^1$ & $C^{g}$ & $\ldots$ & $C^{g^{p-2}}$ \\
        \hline 
        & & & & &\\ 
         & $\left[ \begin{bmatrix}
         1 & 0\\
        0 & 1
    \end{bmatrix} \right]$ & $\left[ \begin{bmatrix}
        1 & b\\
        0 & 1
    \end{bmatrix}\right]$ & $\left [\begin{bmatrix}
        g & b\\
        0 & 1
    \end{bmatrix}\right]$ & $\ldots$ & $\left[\begin{bmatrix}
        g^{p-2} & b\\
        0 & 1
    \end{bmatrix}\right]$\\
     & & & & &\\
	\hline 
 $\chi^0$ & 1 & 1 & 1 & 1 & 1 \\
	\hline
 $\chi^1$ & 1 & 1 & $e^{\frac{2\pi i}{p-1}}$ & $\ldots$ & $e^{\frac{2\pi i(p-2)}{p-1}}$ \\
	\hline $\ldots$ &$\ldots$ & $\ldots$ & $\ldots$ & $\ldots$ & $\ldots$ \\
        \hline
        $\chi^{p-2}$ & 1 & 1 & $e^{\frac{2\pi i(p-2)}{p-1}}$ & \ldots & $e^{\frac{2\pi i(p-2)^2}{p-1}}$ \\
        \hline $\chi^{p-1}$ & $p-1$ & $-1$ & 0 & \ldots & 0 \\
\hline
\end{tabular}
\caption{Character table of $\operatorname{Aff}(\mathbb{F}_p)$}\label{table:chars S_3}
\label{table: chars AffFp}
\end{table} 

This groups, despite not being abelian, have a unique irreducible representation of degree greater than $1$. This makes interpreting its scattering transform easier, thanks to the decomposition provided by the Peter-Weyl theorem (In fact there are only two families of groups with this property \cite{seit}).

\subsubsection{From sounds to functions defined on $\operatorname{Aff}(\mathbb{F}_p)$} \label{Subsec-6-2-2}
To define complex valued functions on $ \operatorname{Aff}(\mathbb{F}_p)$, we use a computational approach inspired by wavelet analysis. 
For $f\in L^2(\mathbb R)$, the continuous wavelet transform (CWT) is defined over $\mathbb{R}^+\times \mathbb{R}$ as a dot product of the function $f$ and wavelet $\Psi$ dilated and translated:
\begin{equation} \label{CWT}
 W f(\delta, s) = \int_{\mathbb R} f(t) \frac{1}{\sqrt{\delta}} \overline{\Psi\left(\frac{t-s}{\delta}\right)}\, dt\,.
\end{equation}
The wavelet we use is the Morlet wavelet cmorB-C from the module pyct.cwt from PyWavelets  (with $B=3.5$ and $C=1.5$), which is
\begin{equation} \label{eq:morlet}
\Psi(t) = \frac{1}{\sqrt{\pi B}}\, e^{-\frac{t^2}{B}}\, e^{2 \pi i C t}\,.
\end{equation}
It is easy to see that $\Psi$ has $L^1(\mathbb R)$-norm equal to 1 and  $L^2(\mathbb R)$-norm equal to $\displaystyle \frac{1}{(2\pi B)^{1/4}}.$ Graphs of the real and imaginary parts of $\Psi$ for $B=3.5$ and $C=1.5$ are given in Figure \ref{fig:morlet}

\begin{figure}[h!] 
\centering
\includegraphics[width=0.8\textwidth]{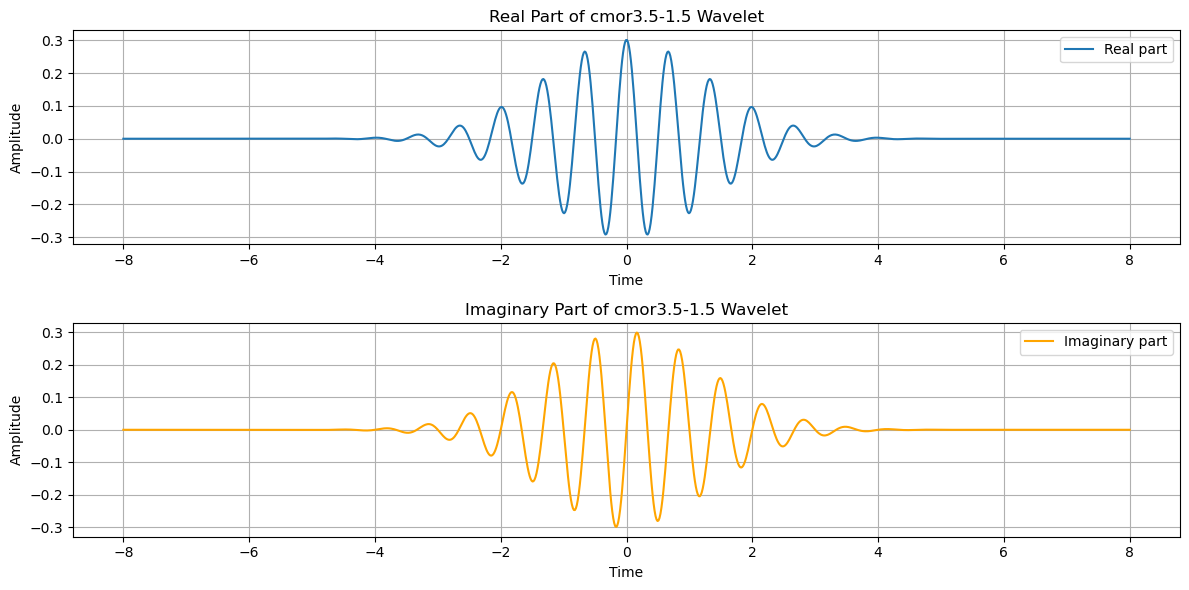} 
\caption{Real (top) and imaginary (bottom) parts of Morlet wavelet with $B=3.5$ and $C=1.5$} 
\label{fig:morlet} 
\end{figure}


Our sounds are discrete functions that take $T=32,000$ values. The discretization of \eqref{CWT} adapted to our setting is 
\begin{equation} \label{CWT-discrete}
	W f(\delta, s) = \frac{2}{T}\sum_{t=0}^{T-1} f(t) \frac{1}{\sqrt{\delta}} \overline{\Psi\left(\frac{t-s}{\delta}\right)}\,,
\end{equation}
for $s=0, \dots, T-1$ and $\delta = 1, \dots, T$ . This computation gives a matrix of size $T \times T$. To obtain functions defined on $\operatorname{Aff}(\mathbb{F}_p)$ we extract from this large matrix a submatrix of $(p-1)\times p$ entries. The submatrix is obtained by $p$-downsampling. That is, the values we obtain are 
\begin{equation} \label{CWT-affine}
	W_p f(a, b) = W f(a\lfloor T / p \rfloor \,,\, b\lfloor T/p \rfloor )\,, \quad a=1, \dots, p-1\,, \ b= 0, \dots, p-1\,.
\end{equation}

We will give classifications results for $p=19, 31$, and $61$. The color maps of $10 \cdot \log_{10} |W_p f (a,b)|$ for $p=31$ and each of the preprocessed sounds of Figure \ref{fig:n1} are given in Figure \ref{fig:31}. Similar maps are obtained for $p=19$ and $p=61$ (See Figure \ref{fig:19-31-61} for the image denoted 'Bark 1' in Figure \ref{fig:n1} and all the values of $p$ tested).

%

%

\begin{figure}[h!] 
\centering
\includegraphics[width=\textwidth]{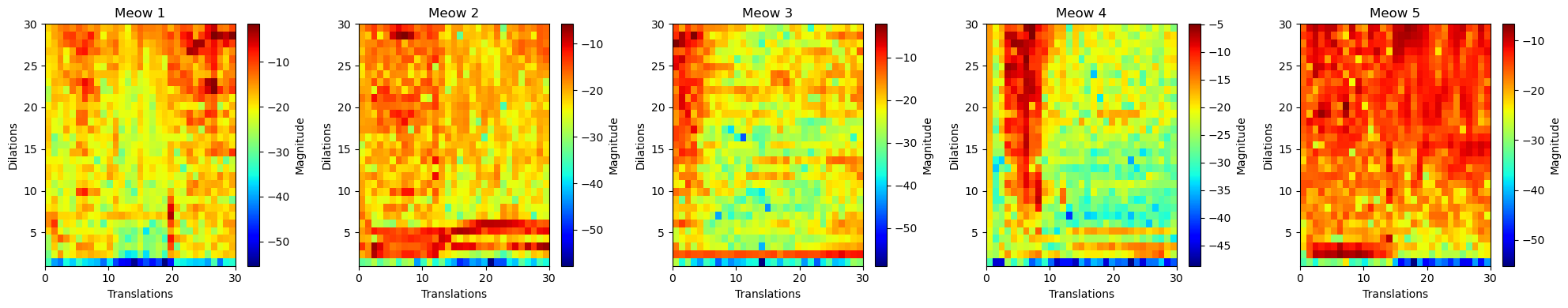} 
\includegraphics[width=\textwidth]{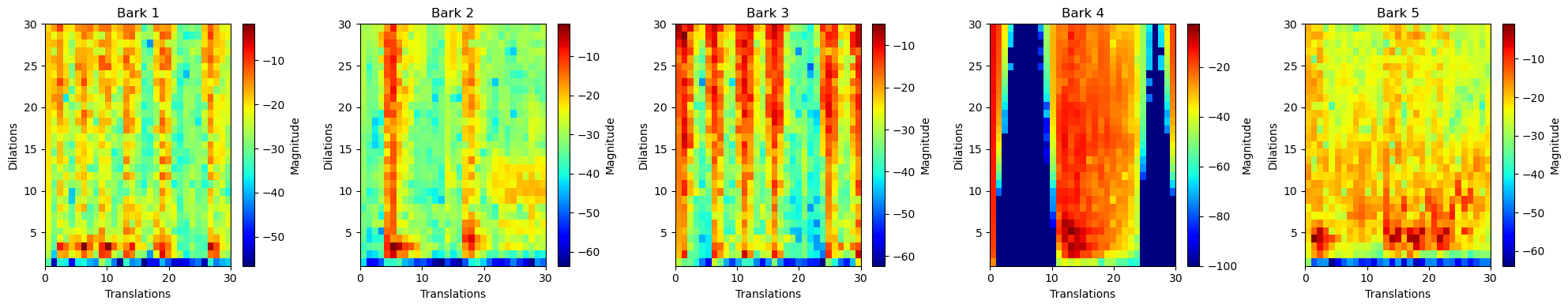} 
\caption{Color maps of $10 \cdot \log_{10} |W_p f (a,b)|$ for the preprocesed sounds of Figure \ref{fig:n1}  as functions on $Aff(\mathbb{F}_{31})$.}
\label{fig:31} 
\end{figure}

\begin{figure}[h!] 
\centering
\includegraphics[width=\textwidth]{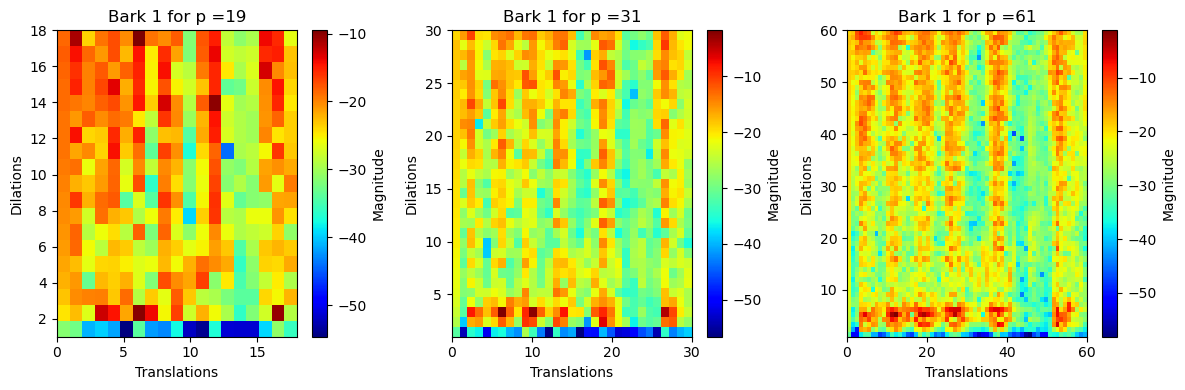} 
\caption{Color maps of $10 \cdot \log_{10} |W_p f (a,b)|$ of the sound denoted `Bark 1' in Figure \ref{fig:n1} as a function on $Aff(\mathbb{F}_{p})$ for different values of $p$.}
\label{fig:19-31-61} 
\end{figure}

Compare Figure \ref{fig:31} with the spectrograms given in Figure \ref{fig:spectrograms}. Both share same similarity in terms of the information they provide about the signal. The discrete wavelet transform \eqref{CWT-affine} captures localized features of the sounds as they evolve over time, similar as how the spectrogram captures energy distribution across different frequencies at different time instants.


\newpage

\subsubsection{The scattering}


Let us denote by $\mathcal S$ de set of all sounds in the dataset (50 meows and 46 barks) after the preprocessing. In subsection \ref{Subsec-6-2-2} we have described how to construct $96$ functions (one for each sound) on $\operatorname{Aff}(\mathbb{F}_{p})$, which we have called $W_p f (a,b), f\in S$ (see \eqref{CWT-affine}).

To perform a scattering transform we choose a kernel $\{\gamma_j: \{0,\dots p-1\} \to \mathbb R \}_{0\leq j \leq 2}$. To learn a kernel we take a training set $\mathcal M$ of 20 meows and a training set $\mathcal B$ of 20 barks, leaving the rest $\mathcal S \setminus (\mathcal M \cup \mathcal B)$ to test de accuracy of the scattering transform with the kernel designed.

For each character $\chi^k\,, \ 0 \leq k \leq p-1$ of $\operatorname{Aff}(\mathbb{F}_{p})$, define
$$
C_{\mathcal B(k)} := \frac{1}{|\mathcal B|} \sum_{f\in \mathcal B} |\langle W_p f \,, \,\chi^k \rangle |\,,
$$
and
$$
C_{\mathcal M(k)}:= \frac{1}{|\mathcal M|} \sum_{f\in \mathcal M} |\langle W_p f \,, \,\chi^k \rangle |\,.
$$
That is, for each training group we perform the arithmetic mean fo the absolute value of the inner product of $W_p f$ and $\chi^k.$

The values of $C_{\mathcal B(k)}$ (red dots) and $C_{\mathcal M(k)}$ (blue dots), $0 \leq k \leq p-1$, for $p=31$ are given in Figure \ref{fig:averagecharacters}\,. Similar results are obtained for $p=19$ and $p=61$. Observe that the values of $C_{\mathcal B(k)}$ and $C_{\mathcal M(k)}$ given in Figure \ref{fig:averagecharacters} are small. The following result gives an estimate for these values.

\begin{figure}
    \centering
    \includegraphics[width=0.85\linewidth]{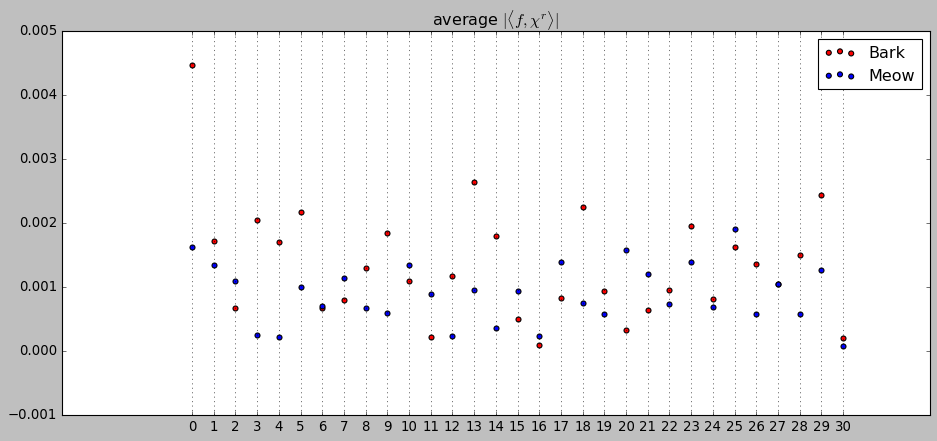}
    \caption{values of $C_{\mathcal B(k)}$ (red dots) and $C_{\mathcal M(k)}$ (blue dots), $0 \leq k \leq p-1$, for $p=31$.}
    \label{fig:averagecharacters}
\end{figure}

\begin{lemma}\label{lem:estimate}
For any $f\in S$, any $p$ prime, and any character $\chi^k\,, 0\leq k \leq p-1$, of $\operatorname{Aff}(\mathbb{F}_{p})$,
$$
|\langle W_p f \,, \, \chi^k \rangle | \leq \frac{1}{\sqrt{2}} \, \frac{1}{(2\pi B)^{1/4}}\,,
$$
where $B$ is the decay of the Morlet wavelet (see \eqref{eq:morlet}). Consequently, for all \, $0\leq k \leq p-1$,
$$
C_{\mathcal B(k)} \leq \frac{1}{\sqrt{2}} \, \frac{1}{(2\pi B)^{1/4}}\,, \qquad \text{and} \qquad C_{\mathcal M(k)} \leq \frac{1}{\sqrt{2}} \, \frac{1}{(2\pi B)^{1/4}}\,.
$$
\end{lemma}

\begin{proof}
By the Cauchy-Schwarz inequality, for all $(a,b)\in \operatorname{Aff}(\mathbb{F}_{p})$,
\begin{align} \label{DB-1}
	|W_p f(a,b)| &= \left|\frac{2}{T}\sum_{t=0}^{T-1} f(t) \frac{1}{\sqrt{a \lfloor T/p \rfloor}}\Psi\left( \frac{t - b \lfloor T/p \rfloor}{a \lfloor T/p \rfloor}  \right)\right| \nonumber \\
	&\leq \left(\frac{2}{T} \sum_{t=0}^{T-1} |f(t)|^2 \, dt \right)^{1/2} \left(\frac{2}{T} \sum_{t=0}^{T-1} \frac{1}{a \lfloor T/p \rfloor} \left|\Psi\left( \frac{t - b \lfloor T/p \rfloor}{a \lfloor T/p \rfloor}  \right) \right|^2 \right)^{1/2}\,.
\end{align}
Since $|f(t)|\leq 1/2$,
\begin{equation*} 
	\left(\frac{2}{T} \sum_{t=0}^{T-1} |f(t)|^2 \, dt \right)^{1/2} \leq \frac{1}{2}\left( \frac{2}{T} T\right)^{1/2} = \frac{1}{\sqrt{2}}\,.
\end{equation*}
After bounding the sum by an integral and doing a change of variables, the second factor of the right hand side of \eqref{DB-1} can be bound by   $ \|\Psi\|_{L^2(\mathbb R)} =\displaystyle \frac{1}{(2\pi B)^{1/4}}.$ Therefore,
\begin{equation} \label{DB-2}
	|W_p f(a,b)| \leq \frac{1}{\sqrt{2}} \frac{1}{(2\pi B)^{1/4}}\,.
\end{equation} 
Use, temporarily, the notation $G= \operatorname{Aff}(\mathbb{F}_{p})$. By the Cauchy-Schwarz inequality and \eqref{DB-2},
	\begin{align} \label{DB-3}
	|\langle W_p f, \chi^k \rangle | &= \left|\frac{1}{|G|} \sum_{(a,b)\in G} W_p f(a,b)\, \overline{\chi^k (a,b)} \right|\nonumber \\
	&\leq \left(\frac{1}{|G|} \sum_{(a,b)\in G} |W_p f(a,b)|^2 \right)^{1/2} \left(\frac{1}{|G|} \sum_{(a,b)\in G} |\chi^k(a,b)|^2 \right)^{1/2} \nonumber \\
	&\leq \frac{1}{\sqrt{2}} \frac{1}{(2\pi B)^{1/4}} \left(\frac{1}{|G|} \sum_{(a,b)\in G} |\chi^k(a,b)|^2 \right)^{1/2} \,.
\end{align}
The last factor in the right hand side of \eqref{DB-3} is $1$ (see \cite[Chapt. 15, Theorem 3]{terras_1999} or use Table \ref{table: chars AffFp}). This proves the result.
\end{proof}

	\begin{remark} \label{Remark-6-1}
		For $B=3.5$ in the Morlet wavelet \eqref{eq:morlet}, the upper bound for $C_{\mathcal B(k)}$ and  $C_{\mathcal M(k)}$ given in Lemma \ref{lem:estimate} is $0,3265$. Figure \ref{fig:averagecharacters} shows smaller values, probably due to the fact that our estimates do not take into account de possible cancellations in the computation of $W_f(a,b).$
	\end{remark}

For $k=0,1, \dots, p-1$, define
\begin{equation*}
	\gamma_1(k)=\begin{cases}
		C_{\mathcal B} (k), \text{ if }  C_{\mathcal B} (k) \geq C_{\mathcal M} (k)\\
		0, \text{ if }  C_{\mathcal B} (k) < C_{\mathcal M} (k),
	\end{cases}
\end{equation*}
and
\begin{equation*}
	\gamma_2(k)=\begin{cases}
		0, \text{ if }  C_{\mathcal B} (k) \geq C_{\mathcal M} (k)\\
		- C_{\mathcal M} (k), \text{ if }  C_{\mathcal B} (k) < C_{\mathcal M} (k).
	\end{cases}
\end{equation*}
By Lemma \ref{lem:estimate},
$$
|\gamma_1(k)|^2 + |\gamma_2(k)|^2 = \max\{|C_{\mathcal B} (k)|^2 , |C_{\mathcal M} (k)|^2 \} \leq  \frac{1}{2} \frac{1}{\sqrt{2\pi B})}\,,
$$
which is smaller than $1$ if $B > 1/(8\pi)$. Finally, for the kernel to satisfy \eqref{multir} we take
$$
\gamma_0 (k) = \sqrt{1 - (|\gamma_1(k)|^2 + |\gamma_2(k)|^2)}\,, \quad k=0, 1, \dots, p-1\,.
$$
Then, we use
$$
\phi(a,b) = \sum_{k=0}^{p-1} \gamma_0(k) \chi^k(a,b)
$$
as low pass filter, and
$$
\psi_j (a,b) = \sum_{k=0}^{p-1} \gamma_j(k) \chi^k(a,b)\,, \quad j=1,2
$$
as wavelets to implement the $G$-scattering transform in $\operatorname{Aff}(\mathbb{F}_{p})$\,.

The $G$-scattering transform is computed with the wavelets just defined over paths of depths $M=0, 1, 2$. For $M=0$ the data is $S[\emptyset] W_p f = \phi\ast W_p f$ which is a function of $p(p-1)$ values. For $M=1$ the data is
$$
S[\emptyset] W_p f\,, \quad S[1] W_p f = \phi\ast |\psi_1\ast W_p f|\,, \quad S[2] W_p f = \phi\ast |\psi_2\ast W_p f|\,.
$$
which is a function of $3p(p-1)$ values. For $M=2$ the data is
$$
S[\emptyset] W_p f\,, \quad S[1] W_p f \,, \quad S[2] W_p f \,, \quad S[1,1] W_p f = \phi\ast |\psi_1\ast|\psi_1\ast W_p f||\,, \quad S[1,2] W_p f = \phi\ast |\psi_2\ast|\psi_1\ast W_p f||\,,
$$
$$
\quad S[2,1] W_p f = \phi\ast |\psi_1\ast|\psi_2\ast W_p f||\,, \quad S[2,2] W_p f = \phi\ast |\psi_2\ast|\psi_2\ast W_p f||\,.
$$

	\begin{remark} \label{Remark-6-2}
		Empirically we observe that the energy of the elements fo the $G$-scattering transform at deeper layers is very small, as predicted by Corollary \ref{corollary2}. For example, for $B=3.5$ and $C=1.5$ in Morlet wavelet \eqref{eq:morlet} the rate of decay estimated with the theoretical bounds of Lemma \ref{lem:estimate} is $\alpha \leq 0.1066$ (see Remark \ref{Remark-6-1}).
	\end{remark}


\subsubsection{Classification and Results}

%
%
%

For classification we use a Support Vector Machine (SVM). The SVM is trained with the data $\mathcal B \cup \mathcal M$ (20 Barks and 20 Meows), and the accuracy is tested on the rest of the sounds of $\mathcal S$.

\


Table \ref{tab:results-Aff} gives the accuracy over the test data of the SVM classifier trained with just the set  $\{ W_p f \}_{\mathcal B \cup \mathcal M}$ with $M=0, 1, 2$ levels as described at the end of the previous subsection and $p=19, 31, 61$.

\begin{table}[h]
    \centering
    \begin{tabular}{|c|c|c|c|c|c|} \hline 
         p&  Only signals on $\operatorname{Aff}(\mathbb F_p)$ &  Scattering 1 Layer&  Scattering 2 Layers \\ \hline 
         19&  41,61\%&  43,75\%&  62,5\%  \\ \hline 
         31&  43,75\%&  68,75\%&  68,75\% \\ \hline 
         61&  50\%&  75\%& \textbf{ 87,5\%} \\ \hline
    \end{tabular}
    \caption{Accuracy of the classification for the Bark and Meow files using G-scattering}
    \label{tab:results-Aff}
\end{table}

\

Although the results are not as good as in the MNIST case (see Table \ref{tab:results}) , the accuracy of the $G$-scattering with 2 layers and $p=61$ ($87,5\%$) is a big improvement over the $41,67\%$ accuracy of SVM applied directly to the data represented on $\operatorname{Aff}(\mathbb{F}_p)$. It may happen that the number of sounds is small. It remains to be tested over a larger amount of sounds. Also, it remains to investigate which type of kernels, for a given group $\operatorname{Aff}(\mathbb{F}_p)$, are more suitable to represent the data.

\subsection{Classification of functions on symmetric groups}\label{symmetricMoreover}

In this section we apply the scattering transform to classify different classes of functions on the symmetric groups $S_3, S_5$ and $S_6$.

\subsubsection{Distinguishing distances on $S_3$ and $S_5$}
In this example, we classify functions arising from various metrics on the symmetric groups $S_3$ and $S_5$. Each class corresponds to a family of functions of the form ${d_j(\pi, \cdot): \tau \mapsto d_j(\pi, \tau) \mid \pi \in S_n}$, where $d_j$ denotes one of the following standard permutation distances (cf. \cite{BUCHHEIM2009962}):
\begin{itemize}
\item $d_1:$ the Hamming distance between two permutations $\pi$ and $\tau$ is the number of different entries, i.e., $ d_1(\pi,\tau)=|\{i \mid \pi(i) \neq \tau(i)\}|$,
\item $d_2:$ the Cayley distance is defined as the minimum number of transpositions taking $\pi$ to $\tau$,
\item $d_3:$ for $p \geq 1$, the $l_p$ distance is defined by $\sqrt[p]{\sum_{i=1}^n|\pi(i)-\tau(i)|^p}$,
\item $d_4:$ the $l_{\infty}$ distance is defined as $\max _{1 \leq i \leq n}|\pi(i)-\tau(i)|$,
\item $d_5:$ the Kendall's tau is the minimum number of pairwise adjacent transpositions taking $\pi$ to $\tau$,
\item $d_6:$ the Ulam's Distance is $n$ minus the length of a longest increasing subsequence in $\left(\tau \pi^{-1}(1), \ldots, \tau \pi^{-1}(n)\right)$.
\item $d_7:$ the Lee Distance is $\sum_{i=1}^n \min (|\pi(i)-\tau(i)|, n-|\pi(i)-\tau(i)|)$
\end{itemize}

The characters are constant in the equivalence classes, which are exactly the cycle types. In the character tables below (see e.g. \cite{GibsonChar}), the columns are labelled by partitions, written in compact notation (eg. $\lambda = \left[4^1 1^1\right]$ means the partition $\lambda = (4, 1)$ of $5$).
\begin{table}[h!]
\centering \label{table: chars S_3}
\begin{tabular}{|c|c|c|c|}
\hline
num elts & 1 & 3 & 2 \\
\hline
class & $\left[1^3\right]$ & $\left[2^1 1^1\right] $ & $\left[3^1\right]$ \\
\hline
$\chi^1 $ & 1 & 1 & 1 \\
\hline
$\chi^2 $ & 2 & 0 & -1 \\
\hline
$\chi^3 $ & 1 & -1 & 1 \\
\hline
\end{tabular}
\end{table}
\begin{table}[h]
\centering
\begin{tabular}{|c|c|c|c|c|c|c|c|}
\hline
\text{num elts} & 1 & 10 & 15 & 20 & 20 & 30 & 24 \\
\hline class & $\left[1^5\right]$ & $\left[2^1  1^3\right]$ & $\left[2^2 1^1\right]$ & $\left[3^1 1^2\right]$ & $\left[3^1 2^1\right]$ & $\left[4^1 1^1\right]$ & $\left[5^1\right]$ \\
\hline $\chi^1$ & 1 & 1 & 1 & 1 & 1 & 1 & 1 \\
\hline$\chi^2$ & 1 & -1 & 1 & 1 & -1 & -1 & 1 \\
\hline$\chi^3$ & 4 & 2 & 0 & 1 & -1 & 0 & -1 \\
\hline$\chi^4$ & 4 & -2 & 0 & 1 & 1 & 0 & -1 \\
\hline$\chi^5$ & 5 & -1 & 1 & -1 & -1 & 1 & 0 \\
\hline$\chi^6$ & 5 & 1 & 1 & -1 & 1 & -1 & 0 \\
\hline$\chi^7$ & 6 & 0 & -2 & 0 & 0 & 0 & 1 \\
\hline
\end{tabular}
\caption{Character Table of $S_3$ and $S_5$}\label{table: chars S_5}
\end{table} 
	
\paragraph{Classification on $S_3$.}
The symmetric group $S_3$, with six elements, is the smallest non-abelian group. The distances from the identity under each of the considered metrics are summarized in Table \ref{tab:s3}. Note that Hamming and Lee distances coincide in this case.

We define six equivalence classes (corresponding to $d_1$ through $d_6$), each containing six functions. Thus, we have $36$ total functions. We partition these evenly into a training set and a test set, assigning three functions per class to each.  We ensure that the canonical representatives $d_j((1), \cdot)$ giving the distance to the identity are included in the training set for each class.
\begin{table}[h]
	\centering
	\begin{tabular}{|c|c|c|c|c|c|c|c|} \hline 
		Element & Hamming  & Cayley  & $\ell_2$  & $\ell_\infty$  & Lee  & Kendall  & Ulam\\ \hline 
		$(1,2,3)$ & 0 &0  &0  &0  &0  &0  &0 \\ \hline 
		$(2,1,3)$& 2 &1  &$\sqrt{2}$  &1  &2  &1  &1 \\ \hline 
		$(3,2,1)$&2  &1  &$2\sqrt{2}$  &  2&2  &3  &2\\ \hline 
		$(1,3,2)$&  2 &1  &$\sqrt{2}$  &  1&2  &1  &1 \\ \hline 
		$(2,3,1)$& 3 &2  &$\sqrt{6}$  &2  &3  &2  &1 \\ \hline 
		$(3,1,2)$&3  &2  &$\sqrt{6}$  &2  &3  &2  &1 \\ \hline
	\end{tabular}
	\caption{Distance to $(1,2,3)$.}
	\label{tab:s3}
\end{table}
To classify these distance functions, we construct a scattering transform with $J = 6$ with depths 1 and 2.  To learn the kernel, we use all the functions from the training set. For $j\in \{1,\ldots, 6\}$, let $$\tilde{d_j}(\cdot)=\frac{1}{\max_{\pi\in S_3}d_j((1),\pi)}d_j((1),\cdot)\,,$$ be the normalized distance to the identity. We define 
\begin{equation}\label{kerj}
\gamma_j(r) = \frac{1}{\sqrt{J}} |\langle \tilde{d_j},\chi^r \rangle|\,.
\end{equation}
Note that by the Cauchy-Schwarz inequality,
$$\sum_{j=1}^6|\gamma_j(r)|^2 = \frac{1}{J}\sum_{j=1}^J|\langle \tilde{d_j},\chi^r \rangle|^2\leq \frac{1}{J}\sum_{j=1}^{J}\|\tilde{d_j}\|^2<1\,,$$ 
so we can define
\begin{equation}\label{ker0}
\gamma_0(r)=\sqrt{1-\sum_{j=1}^6|\gamma_j(r)|^2}. 
\end{equation}

A one-vs-all SVM is trained on the scattering features derived from the training set. The classification accuracy on the test set, using original functions and scattering features of depths 1 and 2, is summarized in Table \ref{tab:accuracy1}.

\begin{table}[h]
    \centering
    \begin{tabular}{|c|c|c|} \hline 
        original funcs & Scattering 1 layer & Scattering 2 layers \\ \hline 
         $9/18$& $15/18$& $16/18$\\ \hline
    \end{tabular}
    \caption{Accuracy on $S_3$}
    \label{tab:accuracy1}
\end{table}

\paragraph{Classification on $S_5$.}
The symmetric group $S_5$ contains $120$ elements. For each of the seven previously defined distances, we construct a corresponding class of functions of the form $\tau \mapsto d_j(\pi, \tau)$, where $\pi \in S_5$ and $d_j$ is one of the distance metrics. This yields a dataset with $7 \cdot 120 = 840$ functions, partitioned into seven function classes of equal size. We divide the dataset evenly into training and test sets, preserving class balance. As before, we ensure that the canonical representatives $d_j((1), \cdot)$, are included in the training set for each class.
We apply the same strategy for feature construction as in the $S_3$ case. That is, we design scattering features derived from the kernel build as in equations \eqref{kerj} and \eqref{ker0} but for $j\in\{1, \ldots, 7\}$.
An SVM classifier is trained using a one-vs-all scheme, where the input features are those obtained from the scattering transform of depth up to 2. The accuracy of the classifier, evaluated on the test set under various configurations: original functions, scattering features with depth 1, and depth 2, is summarized in Table \ref{tab:accuracy2}.

\begin{table}[h]
    \centering
    \begin{tabular}{|c|c|c|} \hline 
        original funcs & Scattering 1 layer & Scattering 2 layers \\ \hline 
         $211/420$ & $336/420$ & $420/420$\\ \hline
    \end{tabular}
    \caption{Accuracy on $S_5$}
    \label{tab:accuracy2}
\end{table}


\subsubsection{Distinguishing random functions on $S_6$}
In this experiment, we consider the classification of left translations of three randomly generated functions on the symmetric group $S_6$. The group $S_6$ contains $720$ elements and has $11$ irreducible representations; its character table is provided in Table \ref{table: chars S_6}.

For each $i \in {1, 2, 3}$, let $f_i : S_6 \to \mathbb{R}$ be a function defined by assigning to each $\pi \in S_6$ a value sampled independently and uniformly from the interval $[-0.5, 0.5]$. The class associated with $f_i$ is then defined as the orbit under left translation: $$C_i = \{L_\tau f_i | \tau \in S_6\}$$
where $(L_\tau f_i)(\pi) := f_i(\tau^{-1} \pi)$. Each such orbit contains $720$ distinct functions, yielding a total of $3 \cdot 720 = 2160$ samples across all classes.

As in previous examples, the dataset is evenly split into training and test sets, maintaining class balance. We also ensure that the original functions $f_1$, $f_2$, and $f_3$ themselves are included in the training set.
\begin{table}[H]
    \centering
\begin{tabular}{|c|c|c|c|c|c|c|c|c|c|c|c|}
\hline Class & {$\left[1^6\right]$} & {$\left[2^1 1^4\right]$} & {$\left[2^2 1^2\right]$} & {$\left[2^3\right]$} & {$\left[3^1 1^3\right]$} & {$\left[3^1 2^1 1^1\right]$} & {$\left[3^2\right]$} & {$\left[4^1 1^2\right]$} & {$\left[4^1 2^1\right]$} & {$\left[5^1 1^1\right]$} & {$\left[6^1\right]$} \\
\hline num elts & 1 & 15 & 45 & 15 & 40 & 120 & 40 & 90 & 90 & 144 & 120 \\
\hline $\chi^1$ & 1 & 1 & 1 & 1 & 1 & 1 & 1 & 1 & 1 & 1 & 1 \\
\hline $\chi^2$ & 5 & 3 & 1 & -1 & 2 & 0 & -1 & 1 & -1 & 0 & -1 \\
\hline $\chi^3$ & 9 & 3 & 1 & 3 & 0 & 0 & 0 & -1 & 1 & -1 & 0 \\
\hline $\chi^4$ & 10 & 2 & -2 & -2 & 1 & -1 & 1 & 0 & 0 & 0 & 1 \\
\hline $\chi^5$ & 5 & 1 & 1 & -3 & -1 & 1 & 2 & -1 & -1 & 0 & 0 \\
\hline $\chi^6$ & 16 & 0 & 0 & 0 & -2 & 0 & -2 & 0 & 0 & 1 & 0 \\
\hline $\chi^7$ & 10 & -2 & -2 & 2 & 1 & 1 & 1 & 0 & 0 & 0 & -1 \\
\hline $\chi^8$ & 5 & -1 & 1 & 3 & -1 & -1 & 2 & 1 & -1 & 0 & 0 \\
\hline $\chi^9$ & 9 & -3 & 1 & -3 & 0 & 0 & 0 & 1 & 1 & -1 & 0 \\
\hline $\chi^{10}$ & 5 & -3 & 1 & 1 & 2 & 0 & -1 & -1 & -1 & 0 & 1 \\
\hline $\chi^{11}$ & 1 & -1 & 1 & -1 & 1 & -1 & 1 & -1 & 1 & 1 & -1 \\
\hline
\end{tabular}
\caption{Character Table of $S_6$}\label{table: chars S_6}
\end{table}

We build a Scattering with $J=3$ up to depth $1$. The kernels are $\gamma_j(r) = |\langle f_j,\chi^r \rangle|$, for $j \in \{1,2,3\}$ and $\gamma_0(r)=\sqrt{1-\sum_{j=1}^3|\gamma_j(r)|^2}$. We employ a one-vs-all support vector machine (SVM) trained on the scattering features to distinguish between the three classes. When evaluated on the test set, the classifier achieves an accuracy of approximately $67.03\%$ using first layer scattering features and $73.33\%$ using second layer scattering features, compared to $40.00\%$ when trained directly on the original function values. These results are summarized in Table~\ref{tab:accuracy3}.

\begin{table}[H]
    \centering
    \begin{tabular}{|c|c|c|} \hline 
        original funcs & Scattering 1 layer & Scattering 2 layers \\ \hline 
         $432/1080$ & $724/1080$ & $792/1080$\\ \hline
    \end{tabular}
    \caption{Accuracy on $S_6$}
    \label{tab:accuracy3}
\end{table}

%
%
%
%
%
%

 
\bibliographystyle{abbrv} 
\bibliography{refs.bib}

\begin{thebibliography}{10}

\bibitem{Anden_2014}
J.~Anden and S.~Mallat.
\newblock Deep scattering spectrum.
\newblock {\em {IEEE} Transactions on Signal Processing}, 62(16):4114--4128,
  2014.

\bibitem{bruna2013classification}
J.~Bruna and S.~Mallat.
\newblock Classification with scattering operators.
\newblock In {\em Proceedings of the {IEEE} Conference on Computer Vision and
  Pattern Recognition ({CVPR})}, pages 1561--1566, Colorado Springs, CO, USA,
  2011.

\bibitem{Bruna_2015}
J.~Bruna, S.~Mallat, E.~Bacry, and J.-F. Muzy.
\newblock Intermittent process analysis with scattering moments.
\newblock {\em The Annals of Statistics}, 43(1):323--351, feb 2015.

\bibitem{BUCHHEIM2009962}
C.~Buchheim, P.~J. Cameron, and T.~Wu.
\newblock On the subgroup distance problem.
\newblock {\em Discrete Mathematics}, 309(4):962--968, 2009.

\bibitem{Burnside1911}
W.~Burnside.
\newblock Theory of groups of finite order.
\newblock {\em The Mathematical Gazette}, 6(95):190–191, 1911.

\bibitem{chew2022geometric}
J.~Chew, M.~Hirn, S.~Krishnaswamy, D.~Needell, M.~Perlmutter, H.~Steach,
  S.~Viswanath, and H.-T. Wu.
\newblock Geometric scattering on measure spaces.
\newblock {\em Applied and Computational Harmonic Analysis}, 70:101635, 2024.

\bibitem{maggioni}
R.~R. Coifman and M.~Maggioni.
\newblock Diffusion wavelets.
\newblock {\em Applied and Computational Harmonic Analysis}, 21(1):53--94,
  2006.
\newblock Special Issue: Diffusion Maps and Wavelets.

\bibitem{cotter2017visualizing}
F.~Cotter and N.~Kingsbury.
\newblock Visualizing and improving scattering networks.
\newblock In {\em 2017 IEEE 27th International Workshop on Machine Learning for
  Signal Processing (MLSP)}, pages 1--6, 2017.

\bibitem{dixon_1967}
J.~D. Dixon.
\newblock High speed computation of group characters.
\newblock {\em Numerische Mathematik}, 10(5):446–450, 1967.

\bibitem{dixon_1970}
J.~D. Dixon.
\newblock Computing irreducible representations of groups.
\newblock {\em Mathematics of Computation}, 24(111):707–712, 1970.

\bibitem{folland_1995}
G.~B. Folland.
\newblock {\em A course in abstract harmonic analysis}.
\newblock CRC Press, 1995.

\bibitem{gauthier2022parametric}
S.~Gauthier, B.~Th\'erien, L.~Als\`ene-Racicot, M.~Chaudhary, I.~Rish,
  E.~Belilovsky, M.~Eickenberg, and G.~Wolf.
\newblock Parametric scattering networks.
\newblock In {\em Proceedings of the IEEE/CVF Conference on Computer Vision and
  Pattern Recognition (CVPR)}, pages 5749--5758, June 2022.

\bibitem{Geller11}
D.~Geller and I.~Pesenson.
\newblock Band-limited localized {P}arseval frames and {B}esov spaces on
  compact homogeneous manifolds.
\newblock {\em Journal of Geometric Analysis}, 21:334--371, 2011.

\bibitem{GibsonChar}
J.~Gibson.
\newblock Characters of the symmetric group.
\newblock https://www.jgibson.id.au/articles/characters/.

\bibitem{grove_1997}
L.~C. Grove.
\newblock {\em Groups and Characters}.
\newblock Wiley, 1997.

\bibitem{hammond}
D.~K. Hammond, P.~Vandergheynst, and R.~Gribonval.
\newblock {Wavelets on graphs via spectral graph theory}.
\newblock {\em {Applied and Computational Harmonic Analysis}}, 30(2):129--150,
  Mar. 2011.

\bibitem{KondorTrivedi}
R.~Kondor and S.~Trivedi.
\newblock On the generalization of equivariance and convolution in neural
  networks to the action of compact groups.
\newblock In J.~Dy and A.~Krause, editors, {\em Proceedings of the 35th
  International Conference on Machine Learning}, volume~80 of {\em Proceedings
  of Machine Learning Research}, pages 2747--2755. PMLR, 10--15 Jul 2018.

\bibitem{kueh_olson_rockmore_tan_2001}
K.~L. Kueh, T.~Olson, D.~Rockmore, and K.~S. Tan.
\newblock Nonlinear approximation theory on compact groups.
\newblock {\em The Journal of Fourier Analysis and Applications},
  7(3):257–281, 2001.

\bibitem{lecun1998gradient}
Y.~LeCun, L.~Bottou, Y.~Bengio, and P.~Haffner.
\newblock Gradient-based learning applied to document recognition.
\newblock In {\em Proceedings of the IEEE}, volume 86(11), pages 2278--2323,
  1998.

\bibitem{Cayley}
L.~Lov\'{a}sz.
\newblock Spectra of graphs with transitive groups.
\newblock {\em Period. Math. Hungar.}, 6(2):191--195, 1975.

\bibitem{Mallat11}
S.~Mallat.
\newblock Group invariant scattering.
\newblock {\em Communications on Pure and Applied Mathematics},
  65(10):1331--1398, 2012.

\bibitem{mcfee2015librosa}
B.~McFee, C.~Raffel, D.~Liang, D.~P. Ellis, M.~McVicar, E.~Battenberg, and
  O.~Nieto.
\newblock librosa: Audio and music signal analysis in python.
\newblock In {\em Proceedings of the 14th Python in Science Conference}, pages
  18--25, 2015.

\bibitem{NICOLA2023122}
F.~Nicola and S.~I. Trapasso.
\newblock Stability of the scattering transform for deformations with minimal
  regularity.
\newblock {\em Journal de Mathématiques Pures et Appliquées}, 180:122--150,
  2023.

\bibitem{pmlr-v107-perlmutter20a}
M.~Perlmutter, F.~Gao, G.~Wolf, and M.~Hirn.
\newblock Geometric wavelet scattering networks on compact {R}iemannian
  manifolds.
\newblock In J.~Lu and R.~Ward, editors, {\em Proceedings of The First
  Mathematical and Scientific Machine Learning Conference}, volume 107 of {\em
  Proceedings of Machine Learning Research}, pages 570--604. PMLR, 20--24 Jul
  2020.

\bibitem{PyTorch}
PyTorch.
\newblock
  https://docs.pytorch.org/vision/stable/generated/torchvision.datasets.MNIST.html.

\bibitem{salas2021rotation}
R.~Rodriguez~Salas, E.~Dokladalova, and P.~Dokl{\'a}dal.
\newblock {Rotation invariant CNN using scattering transform for image
  classification}.
\newblock In {\em {IEEE International Conference on Image Processing (ICIP)}},
  Taipei, Taiwan, Sept. 2019.

\bibitem{saito2017underwater}
N.~Saito and D.~S. Weber.
\newblock Underwater object classification using scattering transform of sonar
  signals.
\newblock https://arxiv.org/abs/1707.03133, 2017.

\bibitem{schneider_1990}
G.~J. Schneider.
\newblock Dixon's character table algorithm revisited.
\newblock {\em Journal of Symbolic Computation}, 9(5-6):601–606, 1990.

\bibitem{seit}
G.~Seitz.
\newblock Finite groups having only one irreducible representation of degree
  greater than one.
\newblock {\em Proceedings of the American Mathematical Society},
  19(2):459--461, 1968.

\bibitem{Sifre2013RotationSA}
L.~Sifre and S.~Mallat.
\newblock Rotation, scaling and deformation invariant scattering for texture
  discrimination.
\newblock {\em 2013 IEEE Conference on Computer Vision and Pattern
  Recognition}, pages 1233--1240, 2013.

\bibitem{stein}
E.~M. Stein.
\newblock {\em Topics in Harmonic Analysis Related to the Littlewood-Paley
  Theory. (AM-63), Volume 63}.
\newblock Princeton University Press, Princeton, 1970.

\bibitem{suh_2019_3563990}
D.~Suh.
\newblock Barkmeowdb - wav files of dogs and cats.
\newblock https://doi.org/10.5281/zenodo.3563990, dec 2019.

\bibitem{terras_1999}
A.~Terras.
\newblock {\em Fourier Analysis on Finite Groups and Applications}.
\newblock London Mathematical Society Student Texts. Cambridge University
  Press, 1999.

\bibitem{waldspurger2016exponential}
I.~Waldspurger.
\newblock Exponential decay of scattering coefficients.
\newblock In {\em 2017 International Conference on Sampling Theory and
  Applications (SampTA)}, pages 143--146, 2017.

\bibitem{wiatowski2017mathematical}
T.~Wiatowski and H.~B{\"o}lcskei.
\newblock A mathematical theory of deep convolutional neural networks for
  feature extraction.
\newblock {\em IEEE Transactions on Information Theory}, 64:1845--1866, 2015.

\bibitem{zou}
D.~Zou and G.~Lerman.
\newblock Graph convolutional neural networks via scattering.
\newblock {\em Applied and Computational Harmonic Analysis}, 49(3):1046--1074,
  2020.

\end{thebibliography}

\end{document}